\documentclass{amsart}

\title[Pluricomplex Green function]{Pluripotential theory on Teichm{\"u}ller space I \\ -- Pluricomplex Green function --}
\author{Hideki Miyachi}

\date{\today}
\address{Division of Mathematical and Physical Sciences,
Graduate School of Natural Science \& Technology,
Kanazawa University,
Kakuma-machi, Kanazawa,
Ishikawa, 920-1192, Japan
}
\email{miyachi@se.kanazawa-u.ac.jp}
\thanks{This work is partially supported by JSPS KAKENHI Grant Numbers
16K05202,
16H03933,
17H02843}
\keywords{Singular Euclidean structures, Teichm{\"u}ller space, Teichm{\"u}ller distance, Levi forms, Pluricomplex Green functions}
\subjclass[2010]{30F60, 32G15, 57M50, 31B05, 32U05, 32U35}

\makeatletter
\@addtoreset{equation}{section}

\makeatother

\usepackage{amsfonts, amsmath, amssymb, amsthm}
\usepackage{enumerate}
\usepackage{url}
\usepackage{graphicx}
\usepackage[dvipdfmx]{color}
\usepackage{amscd}
\newtheorem{theorem}{Theorem}
\theoremstyle{definition}

\newtheorem{example}{Example}[section]

\newtheorem{remark}{Remark}[section]

\newtheorem{lemma}{Lemma}[section]
\newtheorem{proposition}{Proposition}[section]
\newtheorem{corollary}{Corollary}[section]

\newcommand{\teich}{\mathcal{T}}

\newcommand{\ext}{{\rm Ext}}

\newcommand{\order}{\boldsymbol{o}}

\newcommand{\cohomologyclass}[1]{{\bf #1}}
\newcommand{\chainclass}[1]{\widetilde{{\bf #1}}}

\newcommand{\homorel}{{\rm Hom}}

\newcommand{\quaddiff}[2]{{\bf q}[#1,#2]}

\newcommand{\terminaldiff}[2]{{\bf Q}[#1,#2]}

\newcommand{\projQ}{\Pi}

\newcommand{\verticalfoliation}[1]{v(#1)}
\newcommand{\tvector}[2]{{\tt v}\left(#1,#2\right)}
\newcommand{\TeichHomeo}{\Xi}
\newcommand{\teichdisk}{{\bf D}}
\newcommand{\deriv}[2]{{\tt D}_{#1}(#2)}
\newcommand{\derivbar}[2]{\overline{{\tt D}}_{#1}(#2)}
\newcommand{\normalvec}[1]{\boldsymbol{\nu}(#1)}
\newcommand{\Levi}[3]{\mathcal{L}(#1)[#2,#3]}

\makeindex

\begin{document}
 

\begin{abstract}
This is the first paper in a series of investigations of the pluripotential theory on Teichm{\"u}ller space.
One of the main purpose of this paper is to give an alternative approach to the Krushkal formula of the pluricomplex Green function on Teichm{\"u}ller space.
We also show that Teichm{\"u}ller space carries a natural stratified structure of real-analytic submanifolds defined from the structure of singularities of the initial differentials of the Teichm{\"u}ller mappings from a given point.
We will also give a description of the Levi form of the pluricomplex Green function using the Thurston symplectic form via Dumas' symplectic structure on the space of holomorphic quadratic differentials.
\end{abstract}

\maketitle

\section{Introduction}
\label{sec:Introduction}
This is the first paper in a series of investigations of the pluripotential theory on Teichm{\"u}ller space.
One of the purpose of this paper is to give an alternative approach to a characterization of the pluricomplex Green function on Teichm{\"u}ller space, which was first discussed by Krushkal in \cite{Krushkal:1992}:
\begin{theorem}[Pluricomplex Green function on Teichm{\"u}ller space]
\label{thm:main1}
Let $\teich_{g,m}$ be the Teichm{\"u}ller space of Riemann surfaces of analytically finite type $(g,m)$, and $d_T$ the Teichm{\"u}ller distance on $\teich_{g,m}$.
Then, the pluricomplex Green function $g_{\teich_{g,m}}$
on $\teich_{g,m}$ satisfies
\begin{equation}
\label{eq:Klushkal-formula}
g_{\teich_{g,m}}(x,y)=\log\tanh d_T(x,y)
\end{equation}
for $x,y\in \teich_{g,m}$.
\end{theorem}
See \S\ref{subsec:complex-analysis} for the definition of the pluricomplex Green function.
In the second paper \cite{2018arXiv181004343M}, we will establish the Poisson integral formula for pluriharmonic functions on Teichm{\"u}ller space which are continuous on the Bers compactification. The Krushkal formula \eqref{eq:Klushkal-formula} of the pluricomplex Green function plays a crucial rule in the second paper. This result is announced in \cite{Miyachi:2018}.

\subsection{Results}
From Klimek's work \cite{Klimek:1985},
it suffices for proving \eqref{eq:Klushkal-formula} to show that 
the right-hand side of \eqref{eq:Klushkal-formula} is plurisubharmonic
(cf. \S\ref{subsec:complex-analysis}). To show this, Krushkal applied Poletskii's characterization of the pluricomplex Green function (cf. \cite{Poletskii_Shabat:1986}). 

Our strategy is a more direct method with looking ahead to future research (see \S\ref{subsec:background-motivation}). Indeed, we calculate the Levi form of  the log-tanh of the Teichm{\"u}ller distance at generic points, and to check the non-negative definiteness (\S\ref{subsec:log-tanh-PSH}). The Levi form is a fundamental and standard invariant of plurisubharmonic functions in Pluripotential theory (cf. \cite{Klimek:1991}). From our calculation, we observe that the Levi-form of the pluricomplex Green function is described by the Thurston symplectic form on the space of measured foliations via Dumas' symplectic structure on the space of holomorphic quadratic differentials (cf.\cite{Dumas:2015}). This description implies a condition for deformations of Teichm{\"u}ller mappings from a fixed base point to complex-analytically varying targets from the topological aspect in Teichm{\"u}ller theory (cf. \S\ref{sec:Topological-description}). The calculation is established with the variations of the periods of holomorphic one forms on the double covering surfaces defined from initial and terminal Teichm{\"u}ller differentials (cf. \S\ref{sec:pluri-green-function}).

\subsubsection{The Demailly distance on $\teich_{g,m}$}
Let $\Omega$ be a hyperconvex domain in $\mathbb{C}^N$. Let $g_\Omega$ be the pluricomplex Green function on $\Omega$. Demailly \cite[Th{\'e}or{\`e}me 5.3]{Demailly:1987} defined a distance $\boldsymbol{\delta}_\Omega$ related to the pluricomplex Green function by
\begin{equation}
\label{eq:Demailly-distance}
\boldsymbol{\delta}_\Omega(z,w)=\limsup_{\zeta\to \partial \Omega}\left|
\log\frac{g_{\Omega}(z,\zeta)}{g_{\Omega}(w,\zeta)}
\right|\quad (z,w\in \Omega).
\end{equation}
The Demailly distance gives a Harnack-type inequality for the pluriharmonic Poisson kernel.
To discuss the Demailly distance on $\teich_{g,m}$, we identify $\teich_{g,m}$ with the Bers slice with base point $x_0\in \teich_{g,m}$ via the Bers embedding (cf. \cite{Bers:1961}). From Theorem \ref{thm:main1}, we observe the following pluripotential theoretic characterization of the Teichm{\"u}ller distance, which will  be proved in \S\ref{subsec:proof-corollary}.
\begin{corollary}[Demailly distance on $\teich_{g,m}$]
\label{coro:Demailly-distance}
The Demailly distance on $\teich_{g,m}$ coincides with the twice of the Teichm{\"u}ller distance.
\end{corollary}


\subsubsection{Stratification of Teichm{\"u}ller space and Removable singularities}
Dumas \cite{Dumas:2015} gave a complex-analytic stratification in the space $\mathcal{Q}_{x_0}$ of (non-zero) holomorphic quadratic differentials on $x_0\in \teich_{g,m}$ in terms of the structure of singularities (cf. \S\ref{subsec:stratification-Qy0}).
Sending the stratification on $\mathcal{Q}_{x_0}$ by the Teichm{\"u}ller homeomorphism (\S\ref{subsubsec:Teichmuller-homeo}), we obtain a topological stratification on $\teich_{g,m}-\{x_0\}$.
The top stratum $\teich_\infty$ consists of $x\in \teich_{g,m}-\{x_0\}$ such that the initial differential of the Teichm{\"u}ller mapping from $x_0$ to $x$ is generic.
We will show that the induced stratification on $\teich_{g,m}-\{x_0\}$ is a real-analytic stratification in the sense that each stratum is a real-analytic submanifold (cf. Theorem \ref{thm:stratificationofT}).
Applying the stratification, we shall show the following, which is crucial in our proof of Theorem \ref{thm:main1} (cf. \S\ref{subsec:removable-singularites}).

\begin{theorem}[Non-generic strata are removable]
\label{thm:removable-singularities}
A function of class $C^1$ on $\teich_{g,m}-\{x_0\}$ is plurisubharmonic on $\teich_{g,m}$ if 
it is plurisubharmonic on the top stratum $\teich_\infty$
and bounded above around $x_0$.
\end{theorem}

\subsection{Backgrounds, Motivation and Future}
\label{subsec:background-motivation}
Teichm{\"u}ller space $\teich_{g,m}$ is a complex manifold which is homeomorphic to the Euclidean space. The \emph{infinitesimal} complex structure is well-understood from the Kodaira-Spencer theory and the Ahlfors-Bers theory (cf. \cite{Imayoshi_Taniguchi:1992} and \cite{Nag:1988}). Regarding the \emph{global} complex analytic property, it is known that Teichm{\"u}ller space is realized as a polynomially convex and hyperconvex domain in the complex Euclidean space (cf. \cite[Theorem 6.6]{Imayoshi_Taniguchi:1992}, \cite{Shiga:1984} and \cite{Krushkal:1991}), and the Teichm{\"u}ller distance coincides with the Kobayashi distance under the complex structure (cf. \cite{Royden:1971}). 

On the other hand, to the author's knowledge, the \emph{global} complex analytical structure is still less-developed to discuss the \emph{end} (\emph{boundary}) of the Teichm{\"u}ller space from the complex analytical view point. In fact, it is conjectured that the Bers boundary of Teichm{\"u}ller space is a fractal set in some sense (cf. \cite[Question 10.7 in \S10.3]{Canary:2010}). To analyse the geometry of the Bers boundary, we need to understand the behavior of holomorphic functions (holomorphic local coordinates) around the Bers boundary which are defined on the Bers closure, such as the trace functions derived from projective structures or Kleinian groups. Actually, the complex length functions, which are defined from the trace functions, are roughly estimated with topological invariants around the Bers boundary in the proof of the ending lamination theorem (cf. \cite{Minsky:1999}, \cite{Minsky:2010} and \cite{BrockCanaryMinsky:2012}), and the estimations turn out to be very important and useful estimates for studying the boundary (e.g. \cite{KomoriParkkonen:2007} and \cite{Miyachi:2003}). However, it seems to be expected sharper estimates with topological invariants for investigating the geometry of the Bers boundary to establish the conjecture.

According to Thurston theory, the end of the Teichm{\"u}ller space from the topological view point consists of the degenerations of complex or hyperbolic structures on a reference surface via the geometric intersection numbers (cf. \cite{Douady_et_al:1979}). The Teichm{\"u}ller distance and the extremal length are appeared from Quasiconformal geometry on Riemann surfaces. Recently, they are thought of as geometric intersection numbers under the Gardiner-Masur compactification, and give a connection between Quasiconformal geometry and Thurston theory (Extremal length geometry) (cf. \cite{Gardiner_Masur:1991}, \cite{Kerckhoff:1980} and \cite{Miyachi:2014}). 

Pluripotential theory is a powerful theory for investigating complex manifolds. The pluricomplex Green function is one of important functions in Pluripotential theory. The pluricomplex Green function is a fundamental solution of the Dirichlet problem relative to the Monge-Amp\`ere operator, and defines the pluriharmonic measures on the boundaries for hyperconvex domains (cf. \cite{Demailly:1987} and \cite{Klimek:1985}. See also \S\ref{subsec:complex-analysis}). The asymptotic behavior of the pluricomplex Green function on a domain is sensitive in terms of the regularity of the boundary (cf. \cite{Coman:1998}, \cite{DiderichHerbort:2000}, \cite{Herbort:2014}). 

Since Teichm{\"u}ller space is hyperconvex, from Demailly's theory  \cite{Demailly:1987}, $\teich_{g,m}$ admits a unique pluricomplex Green function. By virtue of the ending lamination theorem, the Bers boundary is parametrized by the topological invariants called the ending invariants. By Demailly's Poisson integral formula formulated with the pluriharmonic measures (\cite{Demailly:1987}), holomorphic functions around the Bers closure are represented with their boundary functions which are defined on a space with topological background, and are studied  from a bird's-eye view in Complex function theory.

As a conclusion, the Krushkal formula \eqref{eq:Klushkal-formula} and further investigations on the pluricomplex Green function are expected to strengthen mutual interaction among Quasiconformal geometry, the complex analytic (Pluripotential theoretic) aspect and the topological aspect (Thurston theory) in Teichm{\"u}ller theory.

\subsection{About the paper}
This paper is organized as follows.
From \S\ref{sec:Teichmuller-theory} to \S\ref{sec:stratification-Qgm},
we recall the basic notion and properties in Teichm{\"u}ller theory.
In \S\ref{sec:Deformation-qd},
we discuss the deformation of singular Euclidean structures
associated to the Teichm{\"u}ller deformations from a fixed point $x_0\in \teich_{g,m}$.
In \S\ref{sec:Stratification-Qy0},
we will give the stratification on $\teich_{g,m}-\{x_0\}$.
We show Theorem \ref{thm:main1} and Corollary \ref{coro:Demailly-distance}  in \S\ref{sec:pluri-green-function},
and  discuss the topological description of the Levi form in \S\ref{sec:Topological-description}.

\subsection*{Acknowledgements}
The author thanks Professor Ken'ichi Ohshika for fruitful discussions.
The author also thanks Professor Masanori Adachi for indicating him
to Blanchet's and Chirka's papers \cite{Blanchet:1995} and \cite{Chirka:2003},
and Professor Hiroshi Yamaguchi for his warm advices and encouragements.

\section{Teichm{\"u}ller theory}
\label{sec:Teichmuller-theory}
Let $\Sigma_{g,m}$ be a closed orientable surface 
of genus $g$ with $m$-marked points with $2g-2+m>0$ (possibly $m=0$).
In this section,
we recall basics in Teichm{\"u}ller theory.
For reference,
see \cite{Douady_et_al:1979},
\cite{Gardiner:1987} ,
\cite{Hubbard:2006},
\cite{Imayoshi_Taniguchi:1992},
and \cite{Nag:1988}
for instance.

\subsection{Teichm{\"u}ller space}
\emph{Teichm{\"u}ller space} $\teich_{g,m}$
is the set of equivalence classes
of marked Riemann surfaces of type $(g,m)$. 
A \emph{marked Riemann surface} $(M,f)$ of type $(g,m)$
is a pair of a Riemann surface $M$ of analytically finite type $(g,m)$
and an orientation preserving homeomorphism
$f\colon \Sigma_{g,m}\to M$.
Two marked Riemann surfaces $(M_{1},f_{1})$ and $(M_{2},f_{2})$
of type $(g,m)$ are
\emph{(Teichm{\"u}ller) equivalent} if there is a conformal mapping
$h\colon M_{1}\to M_{2}$ such that $h\circ f_{1}$ is homotopic to $f_{2}$.

The \emph{Teichm{\"u}ller distance} $d_T$ is a complete distance
on $\teich_{g,m}$ defined by
$$
d_T(x_1,x_2)=\frac{1}{2}\log \inf_hK(h)
$$
for $x_i=(M_i,f_i)$ ($i=1,2$),
where the infimum runs over all quasiconformal mapping
$h\colon M_1\to M_2$ homotopic to $f_2\circ f_1^{-1}$,
and $K(h)$ is the maximal dilatation of a quasiconformal mapping $h$.

%
\subsection{Quadratic differentials and Infinitesimal complex structure on $\teich_{g,m}$}
For $x=(M,f)\in \teich_{g,m}$,
we denote by $\mathcal{Q}_{x}$ be the complex Banach space
of holomorphic quadratic differentials $q=q(z)dz^{2}$
on $M$ with $L^1$-norm
$$
\|q\|=\int_{M}|q(z)|\frac{\sqrt{-1}}{2}dz\wedge d\overline{z}
<\infty.
$$
From the Riemann-Roch theorem,
the space $\mathcal{Q}_{x}$ is isomorphic to $\mathbb{C}^{3g-3+m}$.
Let
$$
\projQ\colon
\mathcal{Q}_{g,m}=\cup_{x\in \teich_{g,m}}\mathcal{Q}_{x}\to \teich_{g,m}
$$
be the complex vector
bundle of quadratic differentials over $\teich_{g,m}$.
A differential $q\in \mathcal{Q}_{g,m}$ is said to be \emph{generic}
if all zeros are simple and all marked points of the underlying surface
are simple poles of $q$.
Generic differentials are open and dense subset in $\mathcal{Q}_{g,m}$
and in each fiber $\mathcal{Q}_{x}$ for $x\in \teich_{g,m}$.

\subsubsection{Infinitesimal complex structure}
Teichm{\"u}ller space $\teich_{g,m}$ is a complex manifold of dimension
$3g-3+m$.
The infinitesimal complex structure is described as follows:
Let $x=(M,f)\in \teich_{g,m}$.
Let $L^{\infty}(M)$ be the Banach space of measurable $(-1,1)$-forms
$\mu=\mu(z)d\overline{z}/dz$ on $M$
with
$$
\|\mu\|_{\infty}={\rm ess.sup}_{p\in M}|\mu(p)|<\infty.
$$
The holomorphic tangent space $T_{x}\teich_{g,m}$ at $x$ of $\teich_{g,m}$
is described as the quotient space
$$
L^{\infty}(M)/\{\mu\in L^{\infty}(M)\mid
\langle \mu,\varphi\rangle=0,\forall \varphi\in \mathcal{Q}_{x}\},
$$
where
\begin{equation*}
\label{eq:pairing}
\langle \mu,\varphi\rangle=
\int_{M}\mu(z)\varphi(z)\frac{\sqrt{-1}}{2}dz\wedge d\overline{z}.
\end{equation*}
For $v=[\mu]\in T_x\teich_{g,m}$
and $\varphi\in \mathcal{Q}_x$,
the \emph{canonical pairing} between $T_x\teich_{g,m}$
and $\mathcal{Q}_x$ 
is defined by
$$
\langle v,\varphi\rangle=\langle \mu,\varphi\rangle
$$
and,
it induces
an identification between $\mathcal{Q}_x$
and the holomorphic cotangent space $T_x^*\teich_{g,m}$.

\subsubsection{The Teichm{\"u}ller homeomorphism}
\label{subsubsec:Teichmuller-homeo}
Let $\mathcal{UQ}_{x}$ be the unit ball in $\mathcal{Q}_{x}$.
For $q\in \mathcal{UQ}_{x}$,
we define a quasiconformal mapping $f^q$ on $M$
from the Beltrami differential $\|q\|(\overline{q}/|q|)\in L^\infty(M)$.
Then,
$\teich_{g,m}$ is homeomorphic to $\mathcal{UQ}_{x}$ with
\begin{equation}
\label{eq:Teichmuller-homeomorphism}
\TeichHomeo=\TeichHomeo_x\colon \mathcal{UQ}_{x}\ni q\mapsto (f^q(M),f^q\circ f)\in \teich_{g,m}.
\end{equation}
We call the homeomorphism \eqref{eq:Teichmuller-homeomorphism}
the \emph{Teichm{\"u}ller homeomorphism}.
The Teichm{\"u}ller homeomorphism gives a useful representation of the Teichm{\"u}ller distance
as
\begin{equation}
\label{eq:Teichmuller-homeo-distance}
d_T(x,\TeichHomeo_x(q))=\frac{1}{2}\log \frac{1+\|q\|}{1-\|q\|}=\tanh^{-1}(\|q\|)
\end{equation}
for $q\in \mathcal{UQ}_{x}$,

\subsection{Measured foliations}
Let $\mathcal{S}$ be the set of homotopy classes of 
non-trivial and non-peripheral simple closed curves on $\Sigma_{g,m}$.
Let $i(\alpha,\beta)$ denote the \emph{geometric intersection number}
for simple closed curves $\alpha,\beta\in \mathcal{S}$.
Let $\mathcal{WS}=\{t\alpha\mid t\ge 0, \alpha\in \mathcal{S}\}$
be the set of weighted simple closed curves.
The set $\mathcal{S}$
is canonically identified with
a subset of $\mathcal{WS}$
as weight $1$ curves.

We consider an embedding
\begin{equation*}
\label{eq:MFdef}
\mathcal{WS}\ni t\alpha\mapsto [\mathcal{S}\ni \beta
\mapsto t\,i(\alpha,\beta)]\in \mathbb{R}_{\ge 0}^{\mathcal{S}}.
\end{equation*}
We topologize the function space
$\mathbb{R}_{\ge 0}^{\mathcal{S}}$
with the topology of pointwise convergence.
The closure $\mathcal{MF}$ of the image of the embedding is called
the space of \emph{measured foliations} on $\Sigma_{g,m}$.
The space $\mathcal{MF}$ 
is homeomorphic 
to $\mathbb{R}^{6g-6+2m}$,
and
contains the weighted simple closed curves $\mathcal{WS}$
as a dense subset.
The intersection number on $\mathcal{WS}$
is defined by $i(t\alpha,s\beta)=ts\,i(\alpha,\beta)$
for $t\alpha,s\beta\in \mathcal{WS}$.
The intersection number extends continuously 
as a non-negative function $i(\,\cdot\,,\,\cdot\,)$
on $\mathcal{MF}\times \mathcal{MF}$
with $i(F,F)=0$ and $F(\alpha)=i(F,\alpha)$
for $F\in \mathcal{MF}\subset \mathbb{R}_{\ge 0}^{\mathcal{S}}$
and $\alpha\in \mathcal{S}$.

\subsection{Hubbard-Masur differentials and Extremal length}
Let $x=(M,f)\in \teich_{g,m}$.
For $q\in \mathcal{Q}_x$,
the \emph{vertical foliation} $\verticalfoliation{q}$ of $q$ is a measured foliation
defined by
$$
i(\verticalfoliation{q},\alpha)=\inf_{\alpha'\in f(\alpha)}\int_{\alpha'}\left|{\rm Re}(\sqrt{q})\right|
$$
for $\alpha\in \mathcal{S}$.
Hubbard and Masur observed that
the correspondence,
which we call the \emph{Hubbard-Masur homeomorphism},
\begin{equation}
\label{eq:HM-symp}
\mathcal{V}_{x}\colon \mathcal{Q}_{x}\ni q\mapsto \verticalfoliation{q}\in \mathcal{MF}
\end{equation}
is a homeomorphism (cf. \cite{Hubbard_Masur:1979} and Remark \ref{remark:HM-homeo}).
For $F\in \mathcal{MF}$,
the \emph{Hubbard-Masur differential} $q_{F,x}$ for $F$ at $x$ is defined to
satisfy $\verticalfoliation{q_{F,x}}=F$.
By definition,
$q_{tF,x}=t^2q_{F,x}$ for $F\in \mathcal{MF}$ and $t\ge 0$.

For $F\in \mathcal{MF}$,
the \emph{extremal length} of $F$ at $x$ is defined by
$$
\ext_x(F)=\|q_{F,x}\|.
$$
Kerckhoff \cite{Kerckhoff:1980} observed that 
the Teichm{\"u}ller distance is expressed as
\begin{equation}
\label{eq:kerckhoff-formula}
d_T(x,y)=\frac{1}{2}\log\sup_{\alpha\in \mathcal{S}}\frac{\ext_x(\alpha)}{\ext_y(\alpha)}
\end{equation}
for $x,y\in \teich_{g,m}$.
This expression is called the \emph{Kerckhoff formula} of the Teichm{\"u}ller distance.

\section{Double covering spaces associated to quadratic differentials}
\label{sec:Double-cover}
%
\subsection{Branched covering spaces}
Let $x_0=(M_0,f_0)\in \teich_{g,m}$.
Let $\Sigma_m=\Sigma_m(x_0)$ be the marked points of $M_0$.
Let $q_{0}\in \mathcal{Q}_{x_0}\subset \mathcal{Q}_{g,m}$.
Let 
$\Sigma_{s}(q_{0})$ be
the set
of singularities
of $q_{0}$.
In accordance with  \cite{Masur_Smillie:1991},
a singular point of $q_{0}$ is called \emph{orientable} if it is of even order,
and \emph{non-orientable} otherwise.
Any orientable singular point  is a zero of $q_{0}$.
Let $\Sigma_{o}=\Sigma_{o}(q_{0})$ (resp. $\Sigma_{e}=\Sigma_{e}(q_{0})$)
be the set of non-orientable (resp. orientable)
zeros of $q_{0}$ (``$o$" and ``$e$" stand for ``odd" and ``even").
The number of non-orientable singularities is always even. 
By definition,
$\Sigma_{s}(q_{0})=\Sigma_{o}(q_{0})\cup \Sigma_{e}(q_{0})$.
We define
\begin{align*}
\Sigma(q_{0})
&=\Sigma_{e}(q_{0})\cup \Sigma_{o}(q_{0})\cup\Sigma_{m}(x_{0}),
\\
\Sigma_{sm}(q_{0})&=\Sigma_{s}(q_{0})\cap
\Sigma_{m}(x_{0}),
\\
\Sigma_{s\setminus m}(q_{0})&=\Sigma_{s}(q_{0})\setminus 
\Sigma_{m}(x_{0}),
\\
\Sigma_{m\setminus s}(q_{0})&=\Sigma_{m}(x_{0})\setminus
\Sigma_{s}(q_{0}),
\\
\Sigma_{ub}(q_{0})
&=\Sigma_{e}(q_{0})\sqcup \Sigma_{m\setminus s}(q_{0}).
\end{align*}
We call a marked point in $\Sigma_{m\setminus s}(q_{0})$ \emph{free}.
The set $\Sigma(q_{0})$ is the totality of marked points
caused by $q_{0}$,
and it is represented as the disjoint unions
\begin{align*}
\Sigma(q_{0})
&=
\Sigma_{ub}(q_{0})\sqcup \Sigma_{o}(q_{0}).
\end{align*}
%

Consider the double branched covering space
$\pi_{q_{0}}\colon \tilde{M}_{q_{0}}\to M_0$ of the square root $\sqrt{q_{0}}$
(cf. Figure \ref{fig:coveringspace}).
\begin{figure}
\centering
\includegraphics[width=8cm]{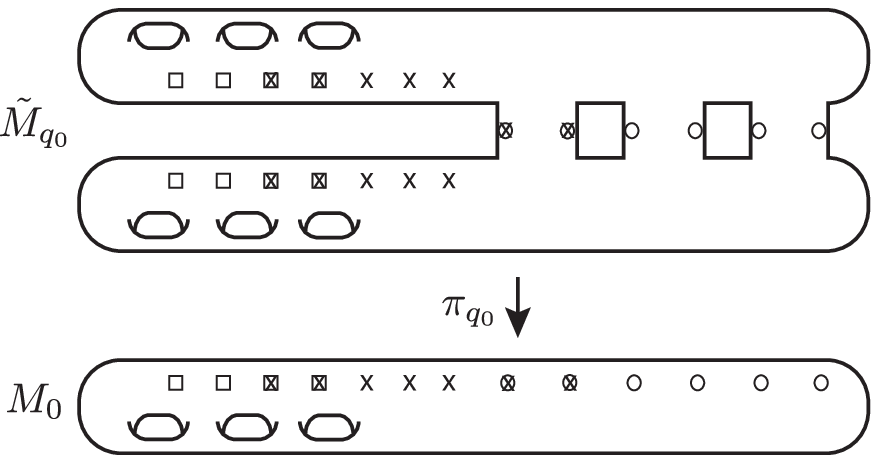}
\caption{
The covering space $\pi_{q_{0}}\colon\tilde{M}_{q_{0}}\to M_{0}$:
Symbols
$\circ$, $\square$ and $\times$ in the figure mean non-orientable singularities,
orientable singularities, and marked points,
respectively.
Each singularity may coincide with a marked point
(denoted by $\otimes$ for non-orientable
singularities and by $\boxtimes$ otherwise).
Points $\otimes$ may or may not be poles of $q_{0}$.
In our notation,
$\Sigma(q_{0})=\{\circ, \otimes, \square,\boxtimes,\times\}$,
$\Sigma_{s}(q_{0})=\{\circ, \otimes, \square,\boxtimes\}$,
$\Sigma_{sm}(q_{0})=\{\otimes,\boxtimes\}$,
$\Sigma_{m\setminus s}(q_{0})=\{\times\}$,
$\Sigma_{s\setminus m}(q_{0})=\{\circ, \square\}$,
$\Sigma_{ub}(q_{0})=\{\times,\square,\boxtimes\}$,
$\Sigma_{o}(q_{0})=\{\circ, \otimes\}$,
and
$\Sigma_{e}(q_{0})=\{\square, \boxtimes\}$.
}
\label{fig:coveringspace}
\end{figure}
For $p\in \Sigma(q_{0})$,
the preimage $\pi_{q_{0}}^{-1}(p)$ consists of two points
if and only if $p\in \Sigma_{ub}(q_{0})$ (``ub'' stands for ``unbranched'').
The projection
$\pi_{q_{0}}$ is double-branched over points in $\tilde{\Sigma}_{o}(q_{0})$.
The surface $\tilde{M}_{q_{0}}$ is a closed surface
of genus
\begin{equation*}
\label{eq:genus-covering}
\tilde{g}(q_{0})=
\begin{cases}
{\displaystyle 2g-1+\frac{1}{2}{}^\#\Sigma_o(q_{0})} & 
(\mbox{$q_{0}$ is not square}) \\
2g &
(\mbox{otherwise}).
\end{cases}
\end{equation*}
(cf. \cite[\S2]{Hubbard_Masur:1979}).
In particular,
the surface $\tilde{M}_{q_{0}}$ is a closed Riemann surface
of genus $4g-3+m$ when $q_{0}$ is generic.
The square root $\sqrt{q_{0}}$ on $M_{0}$ is lifted as the
Abelian differential $\omega_{q_{0}}$
on $\tilde{M}_{q_{0}}$.
The covering transformation $i_{q_{0}}$
of the covering is a conformal involution on $\tilde{M}_{q_{0}}$
which satisfies $\pi_{q_{0}}\circ i_{q_{0}}=\pi_{q_{0}}$
and $i_{q_{0}}^*\omega_{q_{0}}=-\omega_{q_{0}}$.
For each set $\Sigma_{\bullet}(q_{0})$ defined above,
we denote by $\tilde{\Sigma}_{\bullet}(q_{0})$
the preimage of $\Sigma_{\bullet}(q_{0})$.
When $q_0$ is \emph{square} in the sense that $q_0=\omega^2$
for some Abelian differential $\omega$ on $M_0$,
$\tilde{M}_{q_0}$ consists of two copies of $M_0$.
We consider the pair $(\tilde{M}_{q_{0}},\tilde{\Sigma}_{ub}(q_{0}))$
as a Riemann surface with marked points.

\medskip
\noindent
{\bf Convention}\quad
Let $V$ be a vector space $V$ with an involution.
We denote by $V^\pm$ the eigenspace in $V$ of the eigenvalue $\pm 1$
of the action of the involution.\qed

\medskip
We also remark the following elementary fact:
For vector spaces $V_i$ with an involution ($i=1,2,3$),
an exact sequence $0\to V_1\to V_2\to V_3\to 0$
commuting the involutions
induces an exact sequence $0\to V_1^-\to V_2^-\to V_3^-\to 0$.


\subsection{A subspace in quadratic differentials}
For $q_{0}\in\mathcal{Q}_{x_{0}}$ and $x_{0}=(M_{0},f_{0})$,
we define
\begin{align*}
\mathcal{Q}^{T}_{x_{0}}(q_{0})
&=\left\{
\psi\in \mathcal{Q}_{x_{0}}\mid 
(\psi)\ge 
\prod_{p\in \Sigma_{s\setminus m}(q_{0})}
p^{\order_{p}(q_{0})-1}
\prod_{p\in \Sigma_{sm}(q_{0})}
p^{\order_{p}(q_{0})
}
\right\}.
\end{align*}
where $\order_{p}(q_{0})$ is the order of $q_0$ at $p\in M_0$,
and $(\psi)$ is the divisor of $\psi$.
The symbol ``$T$'' stands for ``tangent''.
When $M_{0}$ has no marked point (i.e. $m=0$),
$\psi\in \mathcal{Q}^{T}_{x_{0}}(q_{0})$
is equivalent to the condition that
$\psi/q_{0}$ has at most simple poles on $M_{0}$
(cf. \cite[Lemma 5.2]{Dumas:2015}.
See also Proposition \ref{prop:tangent-space-strata} below).
Notice that 
$\mathcal{Q}_{x_{0}}^T(q_{0})=\mathcal{Q}_{x_{0}}$
if $q_{0}$ is generic.

\subsection{The $q$-realizations of tangent vectors}
Let $x_{0}=(M_0,f_0)\in \teich_{g,m}$ and $q_{0}\in \mathcal{Q}_{x_{0}}-\{0\}$
be a generic differential.
For $v\in T_{x_{0}}\teich_{g,m}$,
a holomorphic quadratic differential
$\eta_v\in \mathcal{Q}_{x_{0}}$ 
said to be the \emph{$q_{0}$-realization} of $v$ 
if it satisfies
\begin{equation}
\label{eq:anti-complex-eta}
\langle v,\psi\rangle_{x_{0}}=\int_{M_{0}}\frac{\overline{\eta_v}}{|q_{0}|}\psi
\end{equation}
for all 
$\psi\in \mathcal{Q}_{x_{0}}$
where $x_{0}=(M_{0},f_{0})\in \teich_{g,m}$
(cf. \cite{Miyachi:2017}).
Notice that
$\eta_v$ in \eqref{eq:anti-complex-eta}
does exist because 
\begin{equation}
\label{eq:Hermitian-product-q0}
(\psi_1,\psi_2)\mapsto \int_{M_{0}}\frac{\psi_1\overline{\psi_2}}{|q_{0}|}
\end{equation}
is a non-degenerate Hermitian inner product on $\mathcal{Q}_{x_{0}}$
(cf. \cite[\S5]{Dumas:2015}).
The correspondence
$$
T_{x_0}\teich_{g,m}\ni v\mapsto \eta_v\in \mathcal{Q}_{x_0}
$$
is an anti-complex linear isomorphism.
The Hermitian form
\eqref{eq:Hermitian-product-q0} is calculated as
\begin{align}
\int_{M_{0}}\frac{\psi_1\overline{\psi_2}}{|q_{0}|}
&=
\frac{\sqrt{-1}}{4}\int_{\tilde{M}_{q_0}}
\frac{\pi_{q_0}^*(\psi_1)}{\omega_{q_0}}\wedge
\overline{\left(
\frac{\pi_{q_0}^*(\psi_2)}{\omega_{q_0}}\right)}
\label{eq:Hermite-form-cover}
\\
&=
\frac{1}{2}\int_{\tilde{M}_{q_0}}
{\rm Re}\left(\frac{\pi_{q_0}^*(\psi_1)}{\omega_{q_0}}
\right)\wedge
{\rm Im}\left(
\frac{\pi_{q_0}^*(\psi_2)}{\omega_{q_0}}\right)
\nonumber
\\
&\qquad+
\frac{\sqrt{-1}}{2}\int_{\tilde{M}_{q_0}}
{\rm Re}\left(\frac{\pi_{q_0}^*(\psi_1)}{\omega_{q_0}}
\right)\wedge
{\rm Re}\left(
\frac{\pi_{q_0}^*(\psi_2)}{\omega_{q_0}}\right)
\nonumber
\end{align}
for $\psi_1$,
$\psi_2\in \mathcal{Q}_{x_0}$.

\section{Stratifications on $\mathcal{Q}_{g,m}$}
\label{sec:stratification-Qgm}

\subsection{Stratification}
Following Dumas \cite{Dumas:2015},
we recall the definition of stratifications on manifolds.
Let $Z$ be a manifold.
A \emph{stratification} of $Z$ is a locally finite collection
of locally closed submanifolds $\{Z_{i}\}_{i\in I}$ of $Z$, the \emph{strata},
indexed by a set $I$ such that
\begin{enumerate}
\item
$Z=\cup_{j\in I}Z_{j}$
\item
$Z_{j}\cap \overline{Z_{k}}\ne \emptyset$
if and only if $Z_{j}\subset \overline{Z_{k}}$.
\end{enumerate}
From the second condition,
$Z_{i}\cap Z_{j}\ne \emptyset$ if and only if $Z_{i}=Z_{j}$
because each $Z_{i}$ is locally closed.
A stratification of a complex manifold $Z$ a
\emph{complex-analytic stratification} if 
the closure $\overline{Z_{j}}$ and 
the boundary $\overline{Z_{j}}\setminus Z_{j}$
of each stratum $Z_{j}$ are complex-analytic sets.

\subsection{Strata  in $\mathcal{Q}_{g,m}$}
\label{subsec:Stratum}
Our strata and symbol are slightly different from that treated by Masur-Smillie \cite{Masur_Smillie:1991} 
and Veech \cite{Veech:1990}.
We consider here the deformation of
\emph{quadratic differentials with marked points}
for our purpose
(see also \S\ref{subsec:Remark-stratification-Qgm} below).
If any marked point of given quadratic differential is a singular point,
our strata are coincides with their strata
(cf. \cite[\S1]{Veech:1990}).

A \emph{symbol} of
$q\in \mathcal{Q}_{g,m}-\{0\}$ is a quadruple
$\boldsymbol{\pi}=(\boldsymbol{m},\boldsymbol{n}(-1),\boldsymbol{n}(\cdot),\varepsilon)$
where $\boldsymbol{m}$
is the number of free marked points,
$\boldsymbol{n}(-1)$ is the number of poles,
$\boldsymbol{n}(l)$ is the number of zeros of order $l\ge 1$,
and $\varepsilon=\pm 1$ according to whether
$q$ is square ($\varepsilon=1$)
or not ($\varepsilon=-1$).
We set $\boldsymbol{n}(0)=0$ for simplicity.
Notice that $\sum_{l\ge -1} l\cdot \boldsymbol{n}(l)=4g-4$.
Let $\mathcal{Q}(\boldsymbol{\pi})=\mathcal{Q}_{g,m}(\boldsymbol{\pi})\subset \mathcal{Q}_{g,m}$
be the set of holomorphic quadratic differentials in $\mathcal{Q}_{g,m}$
whose symbol is $\boldsymbol{\pi}$.
As we discuss in Proposition \ref{prop:local-coordinateQgm} below,
each component of $\mathcal{Q}(\boldsymbol{\pi})$ is a complex manifold of dimension
\begin{align}
\label{eq:dimension-Strata}
\dim_{\mathbb{C}}\mathcal{Q}(\boldsymbol{\pi})
&=
2g+\frac{\varepsilon-3}{2}+\boldsymbol{m}+\sum_{l\ge -1}\boldsymbol{n}(l).
\end{align}
Let ${\boldsymbol \pi}(q)=
(\boldsymbol{m}_{q},\boldsymbol{n}_{q}(-1),\boldsymbol{n}_{q}(\cdot),\varepsilon_{q})$
be the symbol of $q\in \mathcal{Q}_{g,m}$.
If $\Sigma_s(q)=\emptyset$,
we have $g=1$ and $q$ is square.
We set ${\boldsymbol \pi}(q)=(m,0,\{0,\cdots\},1)$ in this case.

Since
$
\boldsymbol{m}_{q}+\sum_{l\ge -1}\boldsymbol{n}_{q}(l)={}^\#\Sigma(q)
={}^\#\Sigma_0(q)+{}^\#\Sigma_{ub}(q)$
for $q\in \mathcal{Q}_{g,m}$,
from 
\eqref{eq:dimension-Strata},
we can check the following.
\begin{align*}
&\dim_{\mathbb{C}}\mathcal{Q}(\boldsymbol{\pi}(q))
=
\dim_{\mathbb{C}}{\rm Hom}(H_1(\tilde{M}_{q},\tilde{\Sigma}_{ub}(q),\mathbb{R})^-,\mathbb{C}).
\label{eq:genus-covering3}
\end{align*}

%
%

\subsection{Remark on the stratification on $\mathcal{Q}_{g,m}$}
\label{subsec:Remark-stratification-Qgm}
Our stratification of $\mathcal{Q}_{g,m}$ is slightly different from 
Masur-Smillie-Veech's one
in the following sense:
We have mainly two differences from their stratification:
\begin{enumerate}
\item
If a free marked point and a singular point collide
in a moving of quadratic differentials,
we recognize the quadratic differentials to be \emph{degenerating} into the other stratum.
Two free marked points can not collide
because we consider the deformation on $\teich_{g,m}$;
and
\item
if a singular point of a quadratic differential in a stratum lies at a marked point,
the singular point stays on the marked point in deforming on the stratum
for the quadratic differential.
\end{enumerate}
All of these phenomena can be handled by standard arguments
with complex analysis
(for instance, \cite{Masur_Smillie:1991} and \cite{Veech:1986}).

\subsection{Masur-Smillie-Veech charts of the strata in $\mathcal{Q}_{g,m}$}
\label{subsec:local-chart-Qpi}
For $q_{0}\in \mathcal{Q}({\boldsymbol \pi}(q_{0}))$,
the union $\cup_{q}\tilde{M}_{q}$
is regarded as a trivial bundle
over a small contractible neighborhood of $q_{0}$
whose fiber is a (possibly disconnected) surface with marked points.
For each $q\in \mathcal{Q}({\boldsymbol \pi}(q_{0}))$
which is sufficiently close to $q_{0}$,
the surface $\tilde{M}_{q}$
admits a marking inherited from
the product structure of the bundle.
Hence,
we can identify $H_1(\tilde{M}_q,\Sigma_{ub}(q),\mathbb{R})^-$
 and $H_1(\tilde{M}_q,\Sigma_{ub}(q),\mathbb{R})$
with 
$H_1(\tilde{M}_{q_0},\Sigma_{ub}(q_{0}),\mathbb{R})^-$
 and $H_1(\tilde{M}_{q_0},\Sigma_{ub}(q_{0}),\mathbb{R})$
for $q\in \mathcal{Q}({\boldsymbol \pi}(q_{0}))$
near $q_{0}$ in the canonical manner.

The following is well-known
(e.g. \cite{Masur_Smillie:1991},
\cite{Masur2:1995},
\cite{Veech:1986} and \cite{Veech:1990}).

\begin{proposition}[Local chart]
\label{prop:local-coordinateQgm}
There is a neighborhood $V_0$ of $q_{0}$ in
$\mathcal{Q}({\boldsymbol \pi}(q_{0}))$
such that
the mapping
\begin{equation*}
\label{eq:local-coordinate}
\Phi_0\colon V_0
\ni q
\mapsto 
\left[
C\mapsto \int_C\omega_{q}
\right]
\in 
{\rm Hom}(H_1(\tilde{M}_{q_0},\Sigma_{ub}(q_{0}),\mathbb{R})^-,\mathbb{C})
\end{equation*}
is a holomorphic local chart around $q_{0}$.
\end{proposition}
%

\section{Deformations of quadratic differentials}
\label{sec:Deformation-qd}
Henceforth,
we set 
$\homorel(q_0)={\rm Hom}(H_1(\tilde{M}_{q_0},\tilde{\Sigma}_{ub}(q_{0}),\mathbb{R})^-,\mathbb{C})$
for the simplicity.
From Proposition \ref{prop:local-coordinateQgm},
$T_{q_0}\mathcal{Q}(\boldsymbol{\pi}(q_0))$
is isomorphic to $\homorel(q_0)$ as $\mathbb{C}$-vector spaces.

In this section, to describe the deformations via the periods geometrically, we consider a $\Delta$-complex structure on $\tilde{M}_{q_0}$ for given $q_0\in \mathcal{Q}_{g,m}$, and describe the infinitesimal deformations along  elements in $T_{q_0}\mathcal{Q}(\boldsymbol{\pi}(q_0))\cong\homorel(q_0)$ by piecewise affine deformations. The description was already discussed by various authors (e.g. \cite{Masur_Smillie:1991} and \cite{Veech:1990}). We also discuss it for the completeness.

\subsection{$\Delta$-complex structure}
A \emph{$\Delta$-complex structure} on a space $X$ is a
collection of a singular simplex $\sigma_\alpha\colon \Delta^n\to X$
($\Delta^n$ is the standard $n$-simplex),
with $n=n(\alpha)$ such that
\begin{enumerate}
\item
the restriction $\sigma_\alpha$ to the interior of $\Delta^n$
is injective,
and each point of $X$ is in the image of exactly one such restriction;
\item
each restriction of $\sigma_\alpha$ to a face of $\Delta^n$ is one of the maps $\sigma_\beta\colon \Delta^{n-1}\to X$.
Here, we are identifying the face of $\Delta^n$ with $\Delta^{n-1}$ by the canonical linear homeomorphism between them
that preserves the ordering of the vertices;
and
\item
a set $A\subset X$ is open if and only if $\sigma^{-1}_\alpha(A)$ is open in $\Delta^n$ for each $\sigma_\alpha$
\end{enumerate}
(cf. \cite[\S2.1]{Hatcher:2002}).
A $\Delta$-complex structure on a surface gives a kind of triangulations.
The (relative) (co)homology group defined by a $\Delta$-complex structure 
on a space $X$ coincides with the (relative) (co)homology group of $X$
(cf. \cite{Hatcher:2002}).

\subsection{Singular Euclidean structure on $\tilde{M}_{q_0}$}
Let $x_0=(M_0,f_0)\in \teich_{g,m}$
and $q_0\in \mathcal{Q}_{x_0}$.
Consider a $\Delta$-complex structure $\Delta$ on $M_0$
such that the $0$-skeleton $\Delta^{(0)}$ contains $\Sigma (q_0)$,
each $1$-simplex is a straight segment
with respect to the $|q_0|$-metric,
and each $2$-simplex is a non-degenerate triangle.
Such a $\Delta$-complex exists.
For instance,
we can take it as a refinement (subdivision) of
the Delaunay triangulation with respect to the singularities of $q_0$
(cf. \cite[\S4]{Masur_Smillie:1991}).
Let $\tilde{\Delta}$ be the lift of $\Delta$.
$\tilde{\Delta}$ is a $\Delta$-complex structure on $\tilde{M}_{q_0}$.
The covering transformation $i_{q_{0}}$ acts on
the $1$-chain group $C_1(\tilde{\Delta},\tilde{\Sigma}_{ub}(q_{0}),\mathbb{R})$.

We define $\chainclass{u}[q_0]\in 
{\rm Hom}(C_1(\tilde{\Delta},\tilde{\Sigma}_{ub}(q_{0}))^-,\mathbb{C})$
by
$$
\chainclass{u}[q_0](e)=\int_e\omega_{q_0}
$$
for $e\in C_1(\tilde{\Delta},\Sigma_{ub}(q_{0}))^-$.
The restriction of $\chainclass{u}[q_0]$ 
to the cycles $Z_1(\tilde{\Delta},\Sigma_{ub}(q_{0}),\mathbb{R})^-$
descends to
a homomorphism $\cohomologyclass{u}[q_0]\in \homorel(q_0)$
such that $\Phi_0(q_0)=\cohomologyclass{u}[q_0]$
(cf. Proposition \ref{prop:local-coordinateQgm}).

\subsection{Piecewise affine deformations}
\label{subsec:small-deformation}
Let $\sigma$ be a $2$-simplex in $\Delta$.
Let $\partial \sigma=e_1+e_2+e_3$ as $1$-chains.
The developing mapping $\sigma\ni p\mapsto z(p)=\int^p \omega_{q_0}$ maps $\sigma$
to a Euclidean triangle $\sigma'$ in the complex plane $\mathbb{C}$
with (oriented) edges $\chainclass{u}[q_0](e_i)$.
Notice that $dz=\omega_{q_0}$ on $\sigma'$
(cf. Figure \ref{fig:Def-triangle}).
\begin{figure}
\centering
\includegraphics[width=5cm]{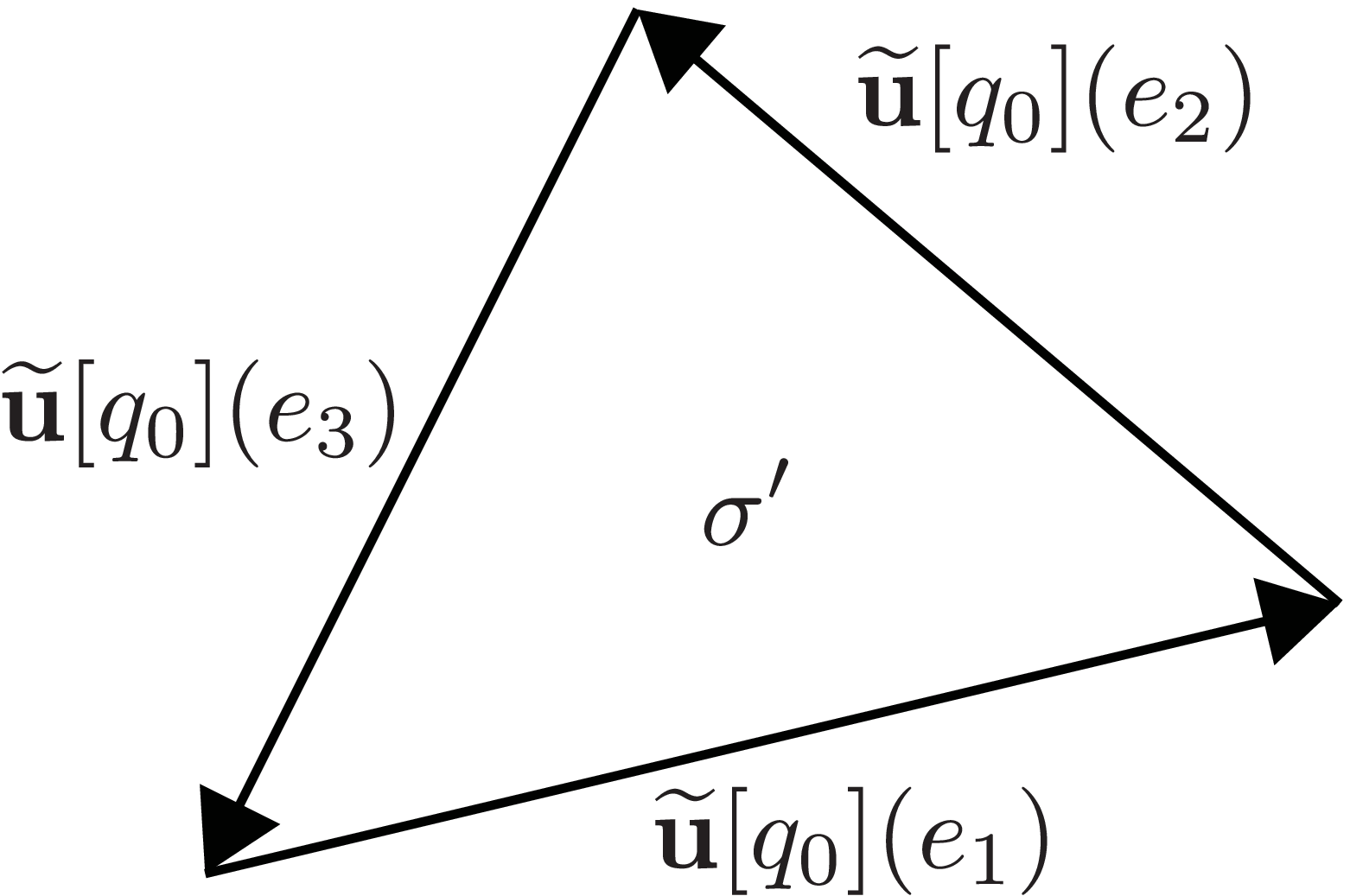}
\caption{Euclidean triangle from $\omega_{q_0}$.}
\label{fig:Def-triangle}
\end{figure}

For $\cohomologyclass{v}\in 
\homorel(q_0)\cong T_{q_0}\mathcal{Q}(\boldsymbol{\pi}(q_0))$,
the infinitesimal deformation along $\cohomologyclass{v}$
of the singular Euclidean structure associated to $q_0$
is described by an assortment of the affine deformation of the triangle $\sigma$
along the lift $\chainclass{v}\in 
{\rm Hom}(C_1(\tilde{\Delta},\tilde{\Sigma}_{ub}(q_{0}),\mathbb{R})^-,\mathbb{C})$ of $\cohomologyclass{v}$.
Here,
we define
the lift $\chainclass{v}$ as follows:
We first take the pullback of $\cohomologyclass{v}$
on $Z_1(\tilde{\Delta},\tilde{\Sigma}_{ub}(q_{0}),\mathbb{R})^-$
by precomposing the projection
from $Z_1(\tilde{\Delta},\tilde{\Sigma}_{ub}(q_{0}),\mathbb{R})^-$
to 
$H_1(\tilde{M}_{q_0},\tilde{\Sigma}_{ub}(q_{0}),\mathbb{R})^-$,
and
set $\chainclass{v}\equiv 0$ on a complementary space of 
$Z_1(\tilde{\Delta},\tilde{\Sigma}_{ub}(q_{0}),\mathbb{R})^-$
in $C_1(\tilde{\Delta},\tilde{\Sigma}_{ub}(q_{0}),\mathbb{R})^-$.


To be more precise,
fix a norm on
${\rm Hom}(C_1(\tilde{\Delta},\tilde{\Sigma}_{ub}(q_{0}),\mathbb{R})^-,\mathbb{C})$.
Since each $2$-simplex of $\Delta$ is a non-degenerate triangle,
vectors
$\{(\chainclass{u}[q_0]+\chainclass{v})(e_i)\}_{i=1}^3$
also span a non-degenerate triangle when
$\chainclass{v}$ is sufficiently short.
Collecting such new triangles defined from all $2$-simplices of $\tilde{\Delta}$,
and gluing them according to the combinatorial structure $\tilde{\Delta}$,
we get a new singular Euclidean surface $\tilde{M}_{q_0}[\cohomologyclass{v}]$
which is homeomorphic to $\tilde{M}_{q_0}$ by a piecewise affine mapping
$\tilde{F}_{\cohomologyclass{v}}\colon \tilde{M}_{q_0}\to 
\tilde{M}_{q_0}[\cohomologyclass{v}]$
defined by assembling the affine deformations on the $2$-simplices of $\tilde{\Delta}$.
Since $i_{q_0}^*(\chainclass{v})=-\chainclass{v}$,
$\tilde{M}_{q_0}[\cohomologyclass{v}]$
admits an involution $i_{q_0}[\cohomologyclass{v}]$
satisfies $i_{q_0}[\cohomologyclass{v}]\circ \tilde{F}_{\cohomologyclass{v}}
=\tilde{F}_{\cohomologyclass{v}}\circ i_{q_0}$,
and
the piecewise affine mapping $\tilde{F}_{\cohomologyclass{v}}$ descends to
a quasiconformal mapping (a piecewise affine mapping)
$F_{\cohomologyclass{v}}$ on $M_0$
to a Riemann surface $M_0[\cohomologyclass{v}]$.
The surface $\tilde{M}_{q_0}[\cohomologyclass{v}]$ admits a $\Delta$-complex structure $\tilde{\Delta}[\cohomologyclass{v}]$
inherited form $\tilde{\Delta}$ on $\tilde{M}_{q_0}$
which is equivariant under the action of the involution $i_{q_0}[\cohomologyclass{v}]$.
The $\Delta$-complex structure $\tilde{\Delta}[\cohomologyclass{v}]$
descends to a $\Delta$-complex structure $\Delta[\cohomologyclass{v}]$
on $M_0[\cohomologyclass{v}]$.
%

Denote by $w$ the flat coordinate for $\tilde{M}_{q_0}[\cohomologyclass{v}]$
(defined on each $2$-simplex of $\tilde{\Delta}[\cohomologyclass{v}]$).
The holomorphic $1$-form $dw$ on each $2$-simplex for 
$\tilde{M}_{q_0}[\cohomologyclass{v}]$
defines a holomorphic $1$-form $\boldsymbol{\omega}_{q_0}[\cohomologyclass{v}]$
on $\tilde{M}_{q_0}[\cohomologyclass{v}]$.
The square $\boldsymbol{\omega}_{q_0}[\cohomologyclass{v}]^2$
descends to a holomorphic quadratic differential
$\quaddiff{q_0}{\cohomologyclass{v}}$
on $M_0[\cohomologyclass{v}]$.

Let $x_{\cohomologyclass{v}}=(M_0[\cohomologyclass{v}],F_{\cohomologyclass{v}}\circ f_0)\in \teich_{g,m}$.
Since $\tilde{F}_{\cohomologyclass{v}}$ sends the vertices of $\tilde{\Delta}$ to
the vertices of $\tilde{\Delta}[\cohomologyclass{v}]$,
we can see that
$\quaddiff{q_0}{\cohomologyclass{v}}\in \mathcal{Q}(\boldsymbol{\pi}(q_0))
\cap \mathcal{Q}_{x_{\cohomologyclass{v}}}$ when $\chainclass{v}$ is sufficiently short,
$\quaddiff{q_0}{0}=q_0$
and
\begin{align}
\label{eq:infinitesimal-1}
\cohomologyclass{u}[\quaddiff{q_0}{\cohomologyclass{v}}](C)=
\int_{(\tilde{F}_{\cohomologyclass{v}})_*(C)}\boldsymbol{\omega}_{q_0}[\cohomologyclass{v}]
=(\cohomologyclass{u}[q_0]+\cohomologyclass{v})(C)
\end{align}
for all $C\in H_1(\tilde{M}_{q_0},\tilde{\Sigma}_{ub},\mathbb{R})^-$.
Summarizing the above argument,
we get a commuting diagram
\begin{equation*}
\label{eq:commute-diagram}
\begin{CD}
\tilde{M}_{q_0} @>{\tilde{F}_{\cohomologyclass{v}}}>>
\tilde{M}_{\quaddiff{q_0}{\cohomologyclass{v}}} \cong \tilde{M}_{q_0}[\cohomologyclass{v}] \\
@VVV @VVV \\
M_0 @>{F_{\cohomologyclass{v}}}>>M_0[\cohomologyclass{v}]
\end{CD}
\end{equation*}
when the lift $\chainclass{v}\in 
{\rm Hom}(C_1(\tilde{\Delta},\tilde{\Sigma}_{ub}(q_{0}),\mathbb{R})^-,\mathbb{C})$ of $\cohomologyclass{v}\in \homorel(q_0)\cong T_{q_0}\mathcal{Q}(\boldsymbol{\pi}(q_0))$ is sufficiently short,
where the vertical directions are double-branched coverings 
with covering involutions $i_{q_0}$ and $i_{q_0}[\cohomologyclass{v}]$.

\subsection{Teichm{\"u}ller deformations}
\label{subsec:Teichmuller-deformation}
Let $x_0=(M_0,f_0)\in \teich_{g,m}$
and $q_0\in \mathcal{Q}_{x_0}$.
%
For $t\ge 0$,
let $h_t\colon M_0\to M_t$ be the Teichm{\"u}ller mapping
associated to the Beltrami differential $\tanh(t)\overline{q_0}/|q_0|$
(cf. \S\ref{subsubsec:Teichmuller-homeo}).
Let $x_{t,q_0}=(M_t,h_t\circ f_0)\in \teich_{g,m}$ and 
Let $\terminaldiff{t}{q_0}\in \mathcal{Q}_{x_{t,q_0}}$ be the terminal differential
(e.g. \cite{Imayoshi_Taniguchi:1992}).
For our purpose, we assume that the Teichm{\"u}ller mapping $h_t$ is represented as an affine mapping associated to the matrix
$\begin{pmatrix} 1 & 0 \\ 0 & e^{-2t} \end{pmatrix}$
in terms of the natural coordinates (distinguished parameters) of the initial and terminal differentials (cf. \cite[Chapter II]{Strebel:1984}).
In particular, $\|\terminaldiff{t}{q_0}\|=e^{-2t}\|q_0\|$ and $\verticalfoliation{\terminaldiff{t}{q_0}}=\verticalfoliation{q_0}$ from our assumption (cf. \cite[Lemma 4.3]{Hubbard_Masur:1979} and \eqref{eq:teichmuller-deformation-homomorphism1} below).
It is known that $\terminaldiff{t}{q_0}\in \mathcal{Q}(\boldsymbol{\pi}(q_0))$ for $t\ge 0$ (e.g. \cite{Masur:1982}).

The Teichm{\"u}ller mapping $h_t$ lifts as a quasiconformal mapping
$\tilde{h}_t\colon \tilde{M}_{q_0}\to \tilde{M}_{q_{t,q_0}}$
which is equivariant under the action of the involutions.
The lift gives the identification
$H_1(\tilde{M}_{\terminaldiff{t}{q_0}},\tilde{\Sigma}_{ub}(\terminaldiff{t}{q_0}),\mathbb{R})^-
\cong H_1(\tilde{M}_{q_0},\tilde{\Sigma}_{ub}(q_0),\mathbb{R})^-$.
By the analytic continuation along a continuous path $t\mapsto \terminaldiff{t}{q_0}\in \mathcal{Q}(\boldsymbol{\pi}(q_0))$ from initial point $q_0$, the chart given in Proposition \ref{prop:local-coordinateQgm} extends a neighborhood of the path.
The image of $\terminaldiff{t}{q_0}$ by the chart satisfies
\begin{align}
\label{eq:teichmuller-deformation-homomorphism1}
\cohomologyclass{u}[\terminaldiff{t}{q_0}](C)
&
=
{\rm Re}\left(
\cohomologyclass{u}[q_{0}](C)
\right)
+\sqrt{-1}e^{-2t}
{\rm Im}\left(
\cohomologyclass{u}[q_{0}](C)
\right)
\\
&=
\frac{1+e^{-2t}}{2}\cohomologyclass{u}[q_{0}](C)
+
\frac{1-e^{-2t}}{2}\overline{\cohomologyclass{u}[q_{0}](C)}
\nonumber
\end{align}
for $C\in H_1(\tilde{M}_{q_0},\tilde{\Sigma}_{ub}(q_0),\mathbb{R})^-$.
 
\subsection{Piecewise affine deformations around Teichm{\"u}ller geodesics}
Let $\cohomologyclass{v}\in \homorel(q_0)\cong T_{q_0}\mathcal{Q}(\boldsymbol{\pi}(q_0))$
and $t>0$.
When the lift $\chainclass{v}$ is sufficiently short as \S\ref{subsec:small-deformation},
we defined
$x_{\cohomologyclass{v}}\in \teich_{g,m}$
and $\quaddiff{q_0}{\cohomologyclass{v}}\in
\mathcal{Q}_{x_\cohomologyclass{v}}$
associated to $\cohomologyclass{v}$
(cf. \S\ref{subsec:small-deformation}).
Consider the Teichm{\"u}ller deformation
on $x_{\cohomologyclass{v}}$ associated to
$\quaddiff{q_0}{\cohomologyclass{v}}$ and $t$.
From the discussion in \S\ref{subsec:Teichmuller-deformation}, the chart in Proposition \ref{prop:local-coordinateQgm} is defined around the terminal differential $\terminaldiff{t}{\quaddiff{q_0}{\cohomologyclass{v}}}$
when the lift of $\cohomologyclass{v}$ is sufficiently short.

From \eqref{eq:infinitesimal-1} and \eqref{eq:teichmuller-deformation-homomorphism1},
$\terminaldiff{t}{\quaddiff{q_0}{0}}
=
\terminaldiff{t}{q_0}$,
$\terminaldiff{0}{\quaddiff{q_0}{\cohomologyclass{v}}}
=
\quaddiff{q_0}{\cohomologyclass{v}}$
and
\begin{align}
\cohomologyclass{u}[\terminaldiff{t}{\quaddiff{q_0}{\cohomologyclass{v}}}](C)
&=
{\rm Re}\left(
(\cohomologyclass{u}[q_{0}]+\cohomologyclass{v})(C)
\right)
+\sqrt{-1}e^{-2t}
{\rm Im}\left(
\cohomologyclass{u}[q_{0}]+\cohomologyclass{v})(C)
\right)
\label{eq:teichmuller-deformation-homomorphism4}
\end{align}
for $t\ge 0$
and
$C\in H_1(\tilde{M}_{q_0},\tilde{\Sigma}_{ub}(q_0),\mathbb{R})^-$
when the lift $\chainclass{v}$
of $\cohomologyclass{v}\in \homorel(q_0)$ is sufficiently short.

%
\subsection{The hypercohomology group}
\label{subsubsec:Hypercohomology-group}
Following Hubbard-Masur \cite{Hubbard_Masur:1979},
we recall the description of the holomorphic
tangent space $T_{q}\mathcal{Q}_{g,m}$
at $q\in \mathcal{Q}_{g,m}$
as the first hypercohomology group $\mathbb{H}^{1}(L^{\bullet})$
a complex of sheaves
(cf. \cite{Godement:1958} or \cite{Griffiths_Harris:1978}).
We will need the Kodaira-Spencer identification of the
tangent space of Teichm{\"u}ller space
with the first cohomology group of the sheaf of holomorphic vector fields
(for instance, see \cite{Kodaira:1986}. See also \cite{Imayoshi_Taniguchi:1992} and \cite{Kawamata:1978}).

Let $X$ and $q$
be a holomorphic vector field 
and a holomorphic quadratic differential
on an open set of a Riemann surface $M$.
Denote by $L_{X}q$ the Lie derivative of $q$ along $X$.
Let $\Theta_{M}$ and $\Omega_{M}^{\otimes 2}$ be the sheaves of germs of
holomorphic vector fields with zeroes at marked points
and meromorphic quadratic differentials on $M$ with (at most)
first order poles at marked points,
respectively.

Let $q_0\in \mathcal{Q}_{g,m}$
and $x_0=(M_0,f_0)\in \teich_{g,m}$ with $q_0\in \mathcal{Q}_{x_0}$
($q_{0}$ need not to be generic).
The tangent space $T_{q_{0}}\mathcal{Q}_{g,m}$
is identified with the 
first hypercohomology group of the complex of sheaves
$$
\begin{CD}
L^{\bullet}\colon\quad
0 @>>> \Theta_{M_0} @>{L_{\cdot}q_{0}}>> \Omega_{M_0}^{\otimes 2}@>>>0.
\end{CD}
$$
The first cochain group is the direct sum
$C^{0}(M_0,\Omega_{M_0}^{\otimes 2})\oplus C^{1}(M_0,\Theta_{M_0})$.
Consider an appropriate covering $\mathcal{U}=\{U_{i}\}_{i}$ on $M_0$
such that $\mathbb{H}^{1}(L^{\bullet})\cong \mathbb{H}^{1}(\mathcal{U},L^{\bullet})$
(see the proof of \cite[Proposition 4.5]{Hubbard_Masur:1979}).

A cochain $(\{\phi_{i}\}_{i},\{X_{ij}\}_{i,j})$
in $C^{0}(\mathcal{U},\Omega_{M_0}^{\otimes 2})\oplus
C^{1}(\mathcal{U},\Theta_{M_0})$
is said to be a \emph{cocycle} if
it satisfies
\begin{equation}
\label{eq:hypercohomology1}
\delta\{X_{ij}\}_{i,j}=X_{ij}+X_{jk}+X_{ki}=0,\quad
\delta\{\phi_{i}\}_{i}=\phi_{i}-\phi_{j}=L_{X_{ij}}(q_{0}).
\end{equation}
A \emph{coboundary} is a cochain $(\{\phi_{i}\}_{i},\{X_{ij}\}_{i,j})$
of the form
\begin{equation}
\label{eq:hypercohomology2}
X_{ij}=Z_{i}-Z_{j}=\delta\{Z_{i}\}_{i},\quad
\phi_{i}=L_{Z_{i}}(q_{0})
\end{equation}
for some $0$-cochain $\{Z_{i}\}_{i}\in C^{0}(\mathcal{U},\Theta_{M_0})$
(cf. Figure \ref{fig:hypercohomology}).
\begin{figure}[t]
\centering
$$
\begin{CD}
0
\\
@AAA \\
C^{0}(\mathcal{U},\Omega_{M_0}^{\otimes 2})
@>{\delta}>>
C^{1}(\mathcal{U},\Omega_{M_0}^{\otimes 2}) \\
@A{L_{\cdot}q_{0}}AA
@A{-L_{\cdot}q_{0}}AA
\\
C^{0}(\mathcal{U},\Theta_{M_0})
@>{\delta}>>
C^{1}(\mathcal{U},\Theta_{M_0})
@>{\delta}>>
C^{2}(\mathcal{U},\Theta_{M_0})
\end{CD}
$$
%
\caption{Double complex for the tangent spaces to $\mathcal{Q}_{g,m}$}
\label{fig:hypercohomology}
\end{figure}

For the hypercohomology class $[(\{\phi_{i}\}_{i},\{X_{ij}\}_{i,j})]
\in \mathbb{H}^{1}(L^{\bullet})$,
when the Kodaira-Spencer class of the 
$1$-cochain $\{X_{ij}\}_{i,j}$ is trivial in $H^1(M_0,\Theta_{M_0})$,
the hypercohomology class $[(\{\phi_{i}\}_{i},\{X_{ij}\}_{i,j})]$
is associated to a holomorphic quadratic differential on $M_0$.
Indeed,
from \eqref{eq:hypercohomology1} and \eqref{eq:hypercohomology2},
$$
\phi_{i}-\phi_{j}=L_{X_{ij}}(q_{0})=L_{Z_i}(q_0)-L_{Z_j}(q_0)
$$
and $\{\phi_{i}-L_{Z_i}(q_0)\}_i$ defines a holomorphic quadratic differential
on $M_0$.

\subsection{Homomorphisms and hypercohomology classes}
From Proposition \ref{prop:local-coordinateQgm},
we have a canonical inclusion
\begin{equation}
\label{eq:homo-hypercohomology}
\homorel(q_0)\cong T_{q_0}\mathcal{Q}(\boldsymbol{\pi}(q))
\hookrightarrow
T_{q_0}\mathcal{Q}_{g,m}\cong
\mathbb{H}^{1}(L^{\bullet}).
\end{equation}
Let $\cohomologyclass{v}\in \homorel(q_0)$
and $[(\{\phi_{i}\}_{i},\{X_{ij}\}_{i,j})]\in \mathbb{H}^{1}(L^{\bullet})$
the corresponding hypercohomology class via \eqref{eq:homo-hypercohomology}.
Take a $0$-cochain
$\{X_{i}\}_{i}$ of the sheaf
of $C^{\infty}$-vector fields such that
$X_{i}-X_{j}=X_{ij}$ on $U_{i}\cap U_{j}$,
and
each $X_i$ vanishes at any marked point of $M_0$.
The $1$-cochain $\{X_{ij}\}_{i,j}$
defines a holomorphic tangent vector at $x_0$
associated to the infinitesimal Beltrami differential
$-(X_i)_{\overline{z}}$ on $M_0$
(cf. \cite[(3.6)]{Miyachi:2017}).
The minus sign comes from our ``$i,j$-convention"
in the definition of the hypercohomology
(compare with Equation $(7.27)$ in \cite[\S7.2.4]{Imayoshi_Taniguchi:1992}).
The holomorphic tangent vector from the $1$-cochain $\{X_{ij}\}_{i,j}$
coincides with the image of $\cohomologyclass{v}\in \homorel(q_0)$
($\hookrightarrow T_{q_0}\mathcal{Q}_{g,m}$)
via the differential of the projection $\mathcal{Q}_{g,m}\to \teich_{g,m}$.

After choosing the covering $\mathcal{U}=\{U_i\}_i$ appropriately,
the right and left sides of the inclusion \eqref{eq:homo-hypercohomology} is 
related to the following formula:
\begin{equation}
\label{eq:differential-cohomology1}
\cohomologyclass{v}(C)=\int_C\Omega[q_0,\cohomologyclass{v}]
\end{equation}
for $C\in H_1(\tilde{M}_{q_0},\tilde{\Sigma}_{ub}(q_0),\mathbb{R})^-$,
where $\Omega[q_0,\cohomologyclass{v}]$ is a $C^\infty$-closed $1$-form
on $\tilde{M}_{q_0}$ defined by
\begin{equation}
\label{eq:differential-cohomology2}
\boldsymbol{\Omega}[q_0,\cohomologyclass{v}]
=
\left(
\frac{\tilde{\phi}_{i}}{2\omega_{q_{0}}}-\omega_{q_{0}}'\tilde{X}_{i}-\omega_{q_{0}}(\tilde{X}_{i})_{z}
\right)dz
-\omega_{q_{0}}
(\tilde{X}_{i})_{\overline{z}}d\overline{z}
\end{equation}
on $U_i$,
and tildes in \eqref{eq:differential-cohomology2}
mean objects (differentials or vector fields etc.)
on the covering space $\tilde{M}_{q_{0}}$
which obtained as lifts of objects on $M_{0}$.
For a proof, see e.g. \cite[Lemma 3.1]{Miyachi:2017}.
Actually,
in \cite[Lemma 3.1]{Miyachi:2017},
we discuss only in the case where $q_0$ is generic.
However, we can deduce \eqref{eq:differential-cohomology2}
the same argument
since we consider deformations of holomorphic
quadratic differentials along strata
as seen in the discussion in \cite[Lemma 5.6]{Dumas:2015}.

\begin{remark}
\label{remark:1}
We notice the following:
\begin{enumerate}
\item
Fix $\epsilon>0$ sufficiently small.
Then,
$\phi_i$ is the $\lambda$-derivative
of the infinitesimal deformation 
of a holomorphic mapping
$\{|\lambda|<\epsilon\}\ni \lambda\mapsto 
\quaddiff{q_0}{\lambda\cohomologyclass{v}}\in \mathcal{Q}(\boldsymbol{\pi}(q_0))$
on $U_i$ at $\lambda=0$
(e.g. \cite[\S3.3]{Miyachi:2017}).
Since 
$\quaddiff{q_0}{\lambda\cohomologyclass{v}}$ varies in $\mathcal{Q}(\boldsymbol{\pi}(q_0))$,
we can see that $\order_p(\phi_i)\ge \order_p(q_0)-1$
for $p\in \Sigma_{s\setminus m}(q_0)\cap U_i$
and 
$\order_p(\phi_i)\ge \order_p(q_0)$
for $p\in \Sigma_{sm}(q_0)\cap U_i$
(cf. \cite[Lemma 5.2]{Dumas:2015}.
Hence the first term of the coefficient of $dz$ of
the differential $\boldsymbol{\Omega}[q_0,\cohomologyclass{v}]$
in \eqref{eq:differential-cohomology2}
is holomorphic around $\tilde{\Sigma}(q_0)\cap U_i$.
\item
For $\cohomologyclass{v}\in \homorel(q_0)$,
we define the \emph{complex conjugate}
$\overline{\cohomologyclass{v}}\in \homorel(q_0)$
of $\cohomologyclass{v}$
by 
$$
\overline{\cohomologyclass{v}}(C)=
\overline{\cohomologyclass{v}(C)}
$$
for $C\in H_1(\tilde{M}_{q_0},\tilde{\Sigma}_{ub}(q_0),\mathbb{R})^-$.
We can easily deduce from \eqref{eq:differential-cohomology1} that
\begin{equation}
\label{eq:differential-cohomology3}
\int_C\boldsymbol{\Omega}[q_0,\overline{\cohomologyclass{v}}]
=
\int_C\overline{\boldsymbol{\Omega}[q_0,\cohomologyclass{v}]}
\end{equation}
for $C\in H_1(\tilde{M}_{q_0},\tilde{\Sigma}_{ub}(q_0),\mathbb{R})^-$.
Notice that the complex conjugate of $\cohomologyclass{v}$ here
is thought of as a tangent vector
in the $(1,0)$-part in the complexification of the real tangent vector
space at $q_0$.
Compare \cite[Proposition 1.5 in Chapter IX]{Kobayashi_Nomizu:1996}.\qed
\end{enumerate}
\end{remark}

We claim 
the following
(cf. \cite{Douady_Hubbard:1975} and \cite[Lemma 5.6]{Dumas:2015}).

\begin{proposition}
\label{prop:trivial-KS-class}
Let $q_0\in \mathcal{Q}_{g,m}$.
Let $\cohomologyclass{v}\in \homorel(q_0)$
and $[\{\phi_i\}_i,\{X_{ij}\}_{i,j})]$ the corresponding hypercohomology class.
When the Kodaira-Spencer class of $\{X_{ij}\}_{i,j}$ is trivial,
\begin{equation*}
\label{eq:Omega-Xij-trivial}
\boldsymbol{\Omega}[q_0,\cohomologyclass{v}]=
\frac{\pi_{q_0}^*(\psi)}{\omega_{q_0}}
\end{equation*}
for some $\psi\in \mathcal{Q}_{x_0}^T(q_0)$
\end{proposition}

\begin{proof}
The assumption implies that there is a $0$-cochain
$\{Z_i\}_i\in C^0(\mathcal{U},\Theta_{M_0})$ such that
$Z_i-Z_j=X_{ij}$.
From \eqref{eq:differential-cohomology2}
we have
\begin{equation*}
\boldsymbol{\Omega}[q_0,\cohomologyclass{v}]
=
\left(
\frac{\tilde{\phi}_{i}}{2\omega_{q_{0}}}-
\omega_{q_{0}}'\tilde{Z}_{i}-\omega_{q_{0}}
\tilde{Z}'_{i}
\right)dz
=
\frac{\tilde{\psi}_i}{2\omega_{q_{0}}}
dz
\end{equation*}
where
\begin{equation}
\label{eq:psi-1}
\psi_i=\phi_i-L_{Z_i}(q_0)=\phi_i-(q_0'Z_i+2q_0Z_i').
\end{equation}
As discussed in the last paragraph of \S\ref{subsubsec:Hypercohomology-group},
$\{\psi_i\}_i$ defines a holomorphic quadratic differential $\psi$ on $M_0$.
We can check from (1) in Remark \ref{remark:1}
and \eqref{eq:psi-1} that $\psi\in \mathcal{Q}_{x_0}^T(q_0)$.
\end{proof}

Let $x_0=(M_0,f_0)\in \teich_{g,m}$.
Suppose $q_0\in \mathcal{Q}_{x_0}$ is generic.
The projection $\projQ\colon \mathcal{Q}_{g,m}\to \teich_{g,m}$ induces a complex linear map
$$
T_{q_0}\mathcal{Q}_{g,m}\cong \homorel(q_0)\ni \cohomologyclass{v}\mapsto
\tvector{\cohomologyclass{v}}{q_0}:=d\projQ\mid_{q_0}[\cohomologyclass{v}]\in T_{x_0}\teich_{g,m}
$$ 
via the differential.
%

\begin{proposition}[Hodge-Kodaira decomposition]
\label{prop:hodge-kodaira-decomposition}
Under the above notation, we have
$$
\cohomologyclass{v}(C)
=
\int_C\frac{
\pi_{q_0}^*(\eta_{\tvector{\overline{\cohomologyclass{v}}}{q_0}})}{\omega_{q_0}}+
\int_C\overline{
\left(
\frac{\pi_{q_0}^*(\eta_{\tvector{\cohomologyclass{v}}{q_0}})}{\omega_{q_0}}
\right)}
$$
for $C\in H_1(\tilde{M}_{q_0},\mathbb{R})^-=
H_1(\tilde{M}_{q_0},\tilde{\Sigma}_{ub}(q_0),\mathbb{R})^-$.
\end{proposition}

\begin{proof}
From the definition of the $q_0$-realizations,
for $\phi\in \mathcal{Q}_{x_0}$,
\begin{align}
\int_{M_0}\mu\phi
&=\int_{M_0}\frac{\overline{\eta_{\tvector{\cohomologyclass{v}}{q_0}}}}{|q_0|}\phi
=\frac{1}{2}
\int_{\tilde{M}_{q_0}}
\frac{\overline{
\pi_{q_0}^*\left(
\eta_{\tvector{\cohomologyclass{v}}{q_0}}\right)}}{|\omega_{q_0}|^2}\pi_{q_0}^*(\phi)
\label{eq:dHKdec1}
\\
&=-\frac{\sqrt{-1}}{4}
\int_{\tilde{M}_{q_0}}
\overline{
\left(\frac{\pi_{q_0}^*\left(
\eta_{\tvector{\cohomologyclass{v}}{q_0}}\right)}{\omega_{q_0}}
\right)}
\wedge \frac{\pi_{q_0}^*(\phi)}{\omega_{q_0}}.
\nonumber
\end{align}
Let $\boldsymbol{\Omega}[q_0,\cohomologyclass{v}]^{(0,1)}$ is the $(0,1)$-part of $\boldsymbol{\Omega}[q_0,\cohomologyclass{v}]$. From \eqref{eq:differential-cohomology2}, a Beltrami differential $\boldsymbol{\Omega}[q_0,\cohomologyclass{v}]^{(0,1)}/\omega_{q_0}$ on $\tilde{M}_{q_0}$ is the lift of the infinitesimal Beltrami differential $\mu$ on $M_0$ associated to $\tvector{\cohomologyclass{v}}{q_0}$. Let $\boldsymbol{\Omega}^h+\overline{\boldsymbol{\Omega}^{ah}}$ be the harmonic form in the de Rham cohomology class of $\boldsymbol{\Omega}[q_0,\cohomologyclass{v}]$, where $\boldsymbol{\Omega}^h$ and $\boldsymbol{\Omega}^{ah}$ are holomorphic $1$-forms on $\tilde{M}_{q_0}$.
Then,
\begin{align}
\int_{M_0}\mu\phi
&=\frac{1}{2}\int_{\tilde{M}_{q_0}}
\frac{\boldsymbol{\Omega}[q_0,\cohomologyclass{v}]^{(0,1)}}{\omega_{q_0}}
\pi_{q_0}^*(\phi) =
-\frac{\sqrt{-1}}{4}\int_{\tilde{M}_{q_0}}
\boldsymbol{\Omega}[q_0,\cohomologyclass{v}]
\wedge
\frac{\pi_{q_0}^*(\phi)}{\omega_{q_0}}
\label{eq:dHKdec2}
\\
&=
-\frac{\sqrt{-1}}{4}\int_{\tilde{M}_{q_0}}
\overline{\boldsymbol{\Omega}^{ah}}
\wedge
\frac{\pi_{q_0}^*(\phi)}{\omega_{q_0}}.
\nonumber
\end{align}
We can easily check that every holomorphic $1$-form in the $(-1)$-eigenspace of
the space of holomorphic $1$-forms
is presented as $\pi_{q_0}^*(\phi)/\omega_{q_0}$
for some $\phi\in \mathcal{Q}_{x_0}$.
From \eqref{eq:dHKdec1} and \eqref{eq:dHKdec2},
we have $\boldsymbol{\Omega}^{ah}=\pi_{q_0}^*\left(
\eta_{\tvector{\cohomologyclass{v}}{q_0}}\right)/\omega_{q_0}$.
Since the harmonic differential in the de Rham cohomology class is unique,
from \eqref{eq:differential-cohomology3},
we deduce that 
$\boldsymbol{\Omega}^{h}=\pi_{q_0}^*\left(
\eta_{\tvector{\overline{\cohomologyclass{v}}}{q_0}}\right)/\omega_{q_0}$.
\end{proof}

For a generic differential $q_0\in \mathcal{Q}_{g,m}$,
we define
$$
\homorel_0(q_0)=\{\cohomologyclass{v}\in \homorel(q_0)\mid
\tvector{\cohomologyclass{v}}{q_0}=0\}.
$$
From Proposition \ref{prop:hodge-kodaira-decomposition},
we have
\begin{corollary}
\label{coro:qd-space-homo}
Let $x_0\in \teich_{g,m}$ and $q_0\in \mathcal{Q}_{x_0}$
a generic differential.
Then,
the mapping
$$
\homorel_0(q_0)\ni \cohomologyclass{v}\mapsto
\eta_{\tvector{\overline{\cohomologyclass{v}}}{q_0}}\in \mathcal{Q}_{x_0}
$$
is a complex linear isomorphism.
\end{corollary}

\begin{example}
\label{example:generic-eta}
For generic $q\in \mathcal{Q}_{g,m}$,
\begin{align*}
\int_C
\frac{\pi_{q}^*(q)}{\omega_{q}}
=
\int_C\omega_{q}
&=\cohomologyclass{u}[q](C)
=\int_C
\frac{\pi_{q}^*(\eta_{\tvector{\overline{\cohomologyclass{u}[q]}}{q}})}
{\omega_{q}}+\overline{
\left(
\frac{\pi_{q}^*(\eta_{\tvector{\cohomologyclass{u}[q]}{q}})}{\omega_{q}}
\right)}
\end{align*}
for $C\in H_1(\tilde{M}_{q},\mathbb{R})^-$.
Therefore,
$\eta_{\tvector{\overline{\cohomologyclass{u}[q]}}{q}}=q$
and
$\eta_{\tvector{\cohomologyclass{u}[q]}{q}}=0$.
\end{example}

\section{Stratification of Teichm{\"u}ller space}
\label{sec:Stratification-Qy0}

\subsection{Stratifications on $\mathcal{Q}_{x_{0}}$}
\label{subsec:stratification-Qy0}
Let $x_{0}\in \teich_{g,m}$.
Dumas \cite{Dumas:2015}	
defined a stratification of $\mathcal{Q}_{x_{0}}$ by symbols
applying the Whitney stratification
(cf. \cite{Schwartz:1966} and \cite{Whitney:1965}).
The stratification on $\mathcal{Q}_{g,m}$
provides a stratification on $\mathcal{Q}_{x_{0}}$
by complex-analytic sets.
This stratification on $\mathcal{Q}_{x_0}$ can be refined as a complex-analytic stratification
$\{Z_{i}\}_{i\in I}=\{Z_{i,x_{0}}\}_{i\in I}$ satisfying the following conditions:
\begin{enumerate}
\item
Each $Z_{i}$ is a complex submanifold of $\mathcal{Q}_{x_{0}}-\{0\}$
invariant under the action of $\mathbb{C}^{*}$;
\item
the symbol is constant on each stratum $Z_{i}$;
\item
$\mathcal{Q}_{x_{0}}-\{0\}=\sqcup_{i\in I}Z_{i}$;
\item
$Z_{i}\cap Z_{j}=\emptyset$ if $i\ne j$; and
\item
the closure $Z_{j}$ is a complex-analytic set,
and if $Z_{i}\cap \overline{Z_{j}}\ne \emptyset$ for $i,j\in I$,
then $Z_{i}\subset \overline{Z_{j}}$.
\end{enumerate}
The refinement refers to changing the stratification
in such a way that each new stratum is entirely contained
in one of the old strata.
Under the situation in the above (5),
$\dim_{\mathbb{C}}Z_{i}<\dim_{\mathbb{C}}Z_{j}$
if $Z_{i}\ne Z_{j}$.
%
There is a unique stratum $Z_{\infty}$ consisting of all generic differentials
(we assume the index set $I$ contains a symbol ``$\infty$'').
Since the stratification is locally finite,
we have

\begin{lemma}
\label{lem:locally-finite}
For any $q_{0}\in Z_{i}-\{0\}$,
there is a neighborhood $U$ in $\mathcal{Q}_{x_{0}}-\{0\}$ of $q_{0}$ such that
$I_{U}=\{i\in I\mid Z_{i}\cap U\ne \emptyset\}$ is a finite set;
and if $Z_{j}\cap U\ne \emptyset$,
then $Z_i\cap \overline{Z_j}\ne \emptyset$. 
\end{lemma}
We extend an observation by Dumas
as follows
(cf. Lemma 5.2 in \cite{Dumas:2015}).

\begin{proposition}[Tangent space to the strata in fibers]
\label{prop:tangent-space-strata}
Let $\{Z_{i}\}_{i\in I}$ is the stratification of $\mathcal{Q}_{y_{0}}$
defined in \S\ref{subsec:stratification-Qy0}.
Let $q_0\in Z_{i}$.
If we identify the tangent space $T_{q_0}Z_{i}$ as a subspace of $\mathcal{Q}_{x_{0}}$,
we have $T_{q_0}Z_{i}\subset \mathcal{Q}_{x_{0}}^{T}(q_0)$.
\end{proposition}

\begin{proof}
Let $p_{0}\in \Sigma_s(q_0)$
and $k_{0}=\order_{p_{0}}(q_0)$.
For simplicity,
$q_{0}$ is assumed to be represented as $q_0=z^{k_{0}}dz^{2}$
around $p_{0}$ with the coordinate $z$ with $z(p_{0})=0$.
From the universal deformation of the singularities,
the deformation of $q_0$ around $p_{0}$
is described as the Lie derivative
$$
L_{X}(z^{k_{0}}dz^{2})=(k_{0}z^{k_{0}-1}X(z)+2z^{k_{0}}X'(z))dz^{2}
$$
along a holomorphic vector field $X=X(z)(\partial/\partial z)$ around $p_{0}$,
where $X(0)=0$ if $p_{0}\in \Sigma_{m}(q_0)$
(cf. \cite[Proposition 3.1]{Hubbard_Masur:1979}.
See also Lemma 5.2 in \cite{Dumas:2015}).
One can see that
the infinitesimal deformation
$\dot{q}$
satisfies
$$
\order_{p_{0}}(\dot{q})\ge
\begin{cases}
k_{0}-1 & (\mbox{if $p_{0}\in \Sigma_{s\setminus m}(q_0)$})\\
k_{0} & (\mbox{if $p_{0}\in \Sigma_{sm}(q_0)$})
\end{cases}
$$
and is contained in 
$ \mathcal{Q}_{x_{0}}^{T}(q_0)$.  
\end{proof}

\subsection{Stratification of Teichm{\"u}ller space}
Let $x_0\in \teich_{g,m}$.
Let $\{Z_i\}_{i\in I}$ be the stratification 
of $\mathcal{Q}_{x_0}-\{0\}$ defined in \S\ref{subsec:stratification-Qy0}.
Let $\mathcal{UQ}_{x_0}$ be the unit ball in $\mathcal{Q}_{x_0}$
with respect to the $L^1$-norm and set
$\TeichHomeo_{x_0}\colon \mathcal{UQ}_{x_0}\to \teich_{g,m}$
be the Teichm{\"u}ller homeomorphism discussed in \S\ref{subsubsec:Teichmuller-homeo}.
For $i\in I$,
we define
$\teich_i=\TeichHomeo_{x_0}(Z_i\cap \mathcal{UQ}_{x_0})$.
The purpose of this section is to show the following.

\begin{theorem}[Stratification]
\label{thm:stratificationofT}
The collection $\{\teich_{i}\}_{i\in I}$ is a stratification of
real-analytic submanifolds
in $\teich_{g,m}-\{x_0\}$.
\end{theorem}
Since $\TeichHomeo_{x_0}$ is a homeomorphism,
$\{\teich_i\}_{i\in I}$ is a stratification of \emph{topological manifolds} in
$\teich_{g,m}-\{x_0\}$.
Namely,
each $\teich_i$ is a locally closed topological submanifold of $\teich_{g,m}-\{0\}$,
the collection $\{\teich_i\}_{i\in I}$ is a locally finite
and satisfies
%
\begin{itemize}
\item[(3')]
$\teich_{g,m}-\{0\}=\sqcup_{i\in I}\teich_{i}$;
\item[(4')]
$\teich_{i}\cap \teich_{j}=\emptyset$ if $i\ne j$;
and
\item[(5')]
if $\teich_{i}\cap \overline{\teich_{j}}\ne \emptyset$ for $i,j\in I$,
then $\teich_{i}\subset \overline{\teich_{j}}$.
\end{itemize}
(The numbers correspond to those in the properties of complex-analytic stratifications given in \S\ref{subsec:stratification-Qy0}.)
We will show that
the restriction of $\TeichHomeo_{x_0}$ to each $Z_i\cap\mathcal{UQ}_{x_0}$
is a real-analytic immersion.
The author does not know if the closure $\overline{\teich_i}$ is an real-analytic subset of $\teich_{g,m}-\{0\}$ for each $i\in I$.
Notice that Theorem \ref{thm:stratificationofT}
is recognized as a  kind of refinements of
Masur's result \cite[Proposition 2.2]{Masur:1995}.

\begin{proof}[Proof of Theorem \ref{thm:stratificationofT}]
Let $i\in I$.
Notice from the definition
that $Z_i$ is a complex submanifold of $\mathcal{Q}_{x_0}$.
From \eqref{eq:Teichmuller-homeo-distance} and
\S\ref{subsec:Teichmuller-deformation},
the Teichm{\"u}ller homeomorphism $\TeichHomeo_{x_0}$ on $Z_i$ is described as
$$
\TeichHomeo_{x_0}(q)=\projQ(\terminaldiff{\tanh^{-1}(\|q\|)}{q})
$$
for $q\in Z_i\cap \mathcal{UQ}_{x_0}$.
From 
Proposition \ref{prop:local-coordinateQgm}
and Riemann's formula,
the norm $Z_i\ni q\mapsto \|q\|$ varies real-analytically
(cf. \cite[\S1]{Kra:1981}
and \cite[Chapter III]{Farkas_Kra:1992}).
Hence the mapping
$$
Z_i\ni q\mapsto \terminaldiff{\tanh^{-1}(\|q\|)}{q}\in \mathcal{Q}(\boldsymbol{\pi}(q_0))
$$
is real-analytic.
Therefore,
$\TeichHomeo_{x_0}$ is also real-analytic on $Z_i$
since $\projQ$ is holomorphic.
Hence,
to complete the proof,
it suffices to show that the (real) differential of the restriction of $\TeichHomeo_{x_0}$ to $Z_i$
is non-singular.

Let $q_0\in Z_i\cap \mathcal{UQ}_{x_0}$ and
$\cohomologyclass{v}\in \homorel(q_0)$
($\hookrightarrow T_{q_0}\mathcal{Q}(\boldsymbol{\pi}(q_0))$)
which tangent to $Z_i$.
For simplicity,
set $Q_0=\terminaldiff{\tanh^{-1}(\|q_0\|)}{q_0}$
and $x_1=(M_1,f_1)=x_{\tanh^{-1}(\|q_0\|),q_0}\in \teich_{g,m}$
(cf. \S\ref{subsec:Teichmuller-deformation}).
From Proposition \ref{prop:tangent-space-strata},
there is $\psi\in \mathcal{Q}^T_{x_0}(q_0)$ such that
\begin{equation}
\label{eq:derivative-period-1}
\cohomologyclass{v}(C)=\int_C
\frac{\pi_{q_0}^*(\psi)}{\omega_{q_0}}
\end{equation}
for $C\in H_1(\tilde{M}_{q_0},\tilde{\Sigma}_{ub}(q_0),\mathbb{R})^-$.

Let $f\colon \{|\lambda|<\epsilon\}\to Z_i$ be a holomorphic disk
with $q(0)=q_0$ and $q'(0)=\cohomologyclass{v}$.
For simplicity,
set $K(t)=\exp(2\tanh^{-1}(\|q(t)\|)=(1+\|q(t)\|)/(1-\|q(t)\|)$ for $t\in \mathbb{R}$ with $|t|<\epsilon$.
From
\eqref{eq:teichmuller-deformation-homomorphism4}
and \eqref{eq:derivative-period-1},
\begin{align}
\label{eq:derivative-period-2}
&
\left.
\frac{d}{dt}
\right|_{t=0}
\cohomologyclass{u}[\terminaldiff{\tanh^{-1}(\|q(t)\|)}{q(t)}](C)
\\
&
=
{\rm Re}
\int_C
\frac{\pi_{q_0}^*(\psi)}{\omega_{q_0}}
+
\sqrt{-1}
{\rm Im}
\left(K(0)^{-1}
\int_C
\frac{\pi_{q_0}^*(\psi)}{\omega_{q_0}}
+(K^{-1})'|_{t=0}\int_C\omega_{q_0}
\right).
\nonumber
\end{align}
for $C\in H_1(\tilde{M}_{q_0},\tilde{\Sigma}_{ub}(q_0),\mathbb{R})^-$.
We denote by $\cohomologyclass{w}(C)$ the right-hand side of \eqref{eq:derivative-period-2}.

Recall that the holomorphic tangent space is the $(1,0)$-part
of the complexification of the real tangent vector space
(cf. \cite[Chapter IX]{Kobayashi_Nomizu:1996}).
In general,
for a complex manifold with a local chart $z=(z_1,\cdots,z_n)$,
a holomorphic tangent vector $\sum_{j=1}^na_j(\partial/\partial z_j)$
is the $(1,0)$-part of a real tangent vector
$\sum_{j=1}^n(a_j(\partial/\partial z_j)+\overline{a_j}(\partial/\partial \overline{z}_j))$ of the underlying differential structure.
The variation \eqref{eq:derivative-period-2} stands for the $(1,0)$-part of the image of the corresponding real tangent vector to $\cohomologyclass{v}$ under the (real) differential of the map $Z_i\ni q\mapsto \terminaldiff{\tanh^{-1}(\|q\|)}{q}\in \mathcal{Q}(\boldsymbol{\pi}(q_0))$ around $q_0$. Then,  $\cohomologyclass{w}$ stands for a homomorphism in $\homorel(Q_0)$ ($\subset T_{Q_0}\mathcal{Q}(\boldsymbol{\pi}(q_0))$) via the isomorphism $H_1(\tilde{M}_{q_0},\tilde{\Sigma}_{ub}(q_0),\mathbb{R})^-\cong H_1(\tilde{M}_{Q_0},\tilde{\Sigma}_{ub}(Q_0),\mathbb{R})^-$ induced by the Teichm{\"u}ller mapping from $x_0$ to $x_1$.

Suppose that the derivative
$$
\{|t|<\epsilon\mid t\in \mathbb{R}\}\ni t\mapsto
\projQ\left(
\terminaldiff{\tanh^{-1}(\|q(t)\|)}{q(t)}
\right)\in \teich_{g,m}
$$
at $t=0$ vanishes.
We will conclude $\cohomologyclass{v}=0$.

Since $\projQ$ is holomorphic,
the differential of $\Pi$
sends the $(1,0)$-part $T_{Q_0}\mathcal{Q}_{g,m}$
of the complexification of the real tangent space at $Q_0$
to that at $x_1$.
From the assumption,
$$
\tvector{\cohomologyclass{w}}{Q_0}=d\projQ\mid_{Q_0}[\cohomologyclass{w}]=0
$$
 in $T_{x_1}\teich_{g,m}$ (cf. \cite[Proposition 2.9, Chapter IX]{Kobayashi_Nomizu:1996}).
From Proposition \ref{prop:trivial-KS-class},
there is $\phi\in \mathcal{Q}^T_{x_1}(Q_0)$ such that
\begin{equation}
\label{eq:derivative-period-3}
\cohomologyclass{w}(C)=\int_C
\frac{\pi_{Q_0}^*(\phi)}{\omega_{Q_0}}
\end{equation}
for $C\in H_1(\tilde{M}_{q_0},\tilde{\Sigma}_{ub}(q_0),\mathbb{R})^-
\cong H_1(\tilde{M}_{Q_0},\tilde{\Sigma}_{ub}(Q_0),\mathbb{R})^-$.

Let $\psi'\in \mathcal{Q}_{x_0}$ and $\phi'\in \mathcal{Q}_{x_1}$
be the holomorphic quadratic differentials
defined by descending the squares $(\pi_{q_0}^*(\psi)/\omega_{q_0})^2$
and $(\pi_{Q_0}^*(\phi)/\omega_{Q_0})^2$
respectively.
Comparing the real parts of \eqref{eq:derivative-period-2} and \eqref{eq:derivative-period-3}
we have $\verticalfoliation{\phi'}=\verticalfoliation{\psi'}$ in $\mathcal{MF}$
(cf. \cite[Lemma 4.3]{Hubbard_Masur:1979}).
Since $K(0)=e^{2d_T(x_0,x_1)}$,
\begin{equation}
\label{eq:stratification-1}
K(0)^{-1}\|\psi'\|=e^{-2d_T(x_0,x_1)}\ext_{x_0}(\verticalfoliation{\psi'})
\le \ext_{x_1}(\verticalfoliation{\phi'})=\|\phi'\|
\end{equation}
from the Kerckhoff formula
(see also \cite[Lemma 4.1]{Gardiner_Masur:1991}).

By Riemann's formula, \eqref{eq:Hermite-form-cover} and \eqref{eq:derivative-period-2},
\begin{align}
\|\phi'\|
&=
\frac{1}{2}
\int_{\tilde{M}_{Q_0}}
{\rm Re}\left(\frac{\pi_{Q_0}^*(\phi)}{\omega_{Q_0}}\right)
\wedge
{\rm Im}\left(
\frac{\pi_{Q_0}^*(\phi)}{\omega_{Q_0}}
\right)
\label{eq:stratification-3}
\\
&=K(0)^{-1}
\frac{1}{2}\int_{\tilde{M}_{q_0}}{\rm Re}
\left(
\frac{\pi_{q_0}^*(\psi)}{\omega_{q_0}}
\right)
\wedge {\rm Im}\left(
\frac{\pi_{q_0}^*(\psi)}{\omega_{q_0}}
\right)
\nonumber
\\
&\qquad
+(K^{-1})'|_{t=0}
\frac{1}{2}\int_{\tilde{M}_{q_0}}{\rm Re}
\left(
\frac{\pi_{q_0}^*(\psi)}{\omega_{q_0}}
\right)
\wedge {\rm Im}(\omega_{q_0})
\nonumber
\\
&=K(0)^{-1}\|\psi'\|
+(K^{-1})'(0)\,
{\rm Re}
\left\{
\frac{\sqrt{-1}}{4}\int_{\tilde{M}_{q_0}}\frac{\pi_{q_0}^*(\psi)}{\omega_{q_0}}\wedge
\overline{\omega_{q_0}}
\right\}.
\nonumber
\end{align}
Since 
\begin{equation}
\label{eq:stratification-2}
\left.
\frac{d\|q(t)\|}{dt}
\right|_{t=0}
=
{\rm Re}
\left\{
\frac{\sqrt{-1}}{2}
\int_{\tilde{M}_{q_0}}\frac{\pi_{q_0}^*(\psi)}{\omega_{q_0}}\wedge
\overline{\omega_{q_0}}
\right\},
\end{equation}
from \eqref{eq:stratification-1}
and \eqref{eq:stratification-3},
we obtain
\begin{align*}
0
&\le \|\phi'\|-K(0)^{-1}\|\psi'\|
=(K^{-1})'(0)\,
{\rm Re}
\left\{
\frac{\sqrt{-1}}{4}\int_{\tilde{M}_{q_0}}\frac{\pi_{q_0}^*(\psi)}{\omega_{q_0}}\wedge
\overline{\omega_{q_0}}
\right\}
\\
&
=
-\frac{1}{(1+\|q_0\|)^2}\left(\left.
\frac{d\|q(t)\|}{dt}
\right|_{t=0}
\right)^2,
\end{align*}
and $(K^{-1})'=0$ at $t=0$.
Therefore,
$\|\phi'\|=K(0)^{-1}\|\psi'\|$
from \eqref{eq:stratification-3} again.
Hence,
we have $\verticalfoliation{\psi'}=s\verticalfoliation{q_0}$ and $\psi'=s^2q_0$
for some $s\ge 0$
from the uniqueness of the extremal problem for the Kerckhoff formula
(or the Teichm{\"u}ller uniqueness theorem. See \cite[Theorem 5.9]{Imayoshi_Taniguchi:1992}
or \cite[\S (3.5)]{Abikoff:1980}).
From \eqref{eq:stratification-2},
we obtain
$$
0=
\left.\frac{d\|q(t)\|}{dt}
\right|_{t=0}
=s^2\|q_0\|
$$
and $s=0$.
Therefore,
$\psi'=0$ and
$\cohomologyclass{v}=0$ from \eqref{eq:derivative-period-1}.
%
\end{proof}

From \eqref{eq:Teichmuller-homeo-distance},
we conclude the following.

\begin{corollary}
\label{coro:Rees}
The Teichm{\"u}ller distance function $\teich_{g,m}-\{x_0\}\ni x\mapsto d_T(x_0,x)$
is real-analytic on each stratum of $\{\teich_{i}\}_{i\in I}$.
\end{corollary}
Recall that the top stratum $Z_\infty$ of the stratification of $\mathcal{Q}_{x_0}$
is an open set which consists of generic differentials.
From Theorem \ref{thm:stratificationofT},
the restriction of the Teichm{\"u}ller homeomorphism $\TeichHomeo_{x_0}\colon Z_\infty
\to \teich_\infty$ is a real-analytic diffeomorphism.
Hence,
Corollary \ref{coro:Rees} is thought of as an extension of an observation by Rees
in \cite[\S2.3]{Rees:2002}.

\subsection{Proof of Theorem 
\ref{thm:removable-singularities}}
\label{subsec:removable-singularites}
%
%
We use the following removable singularity theorem due to Blanchet and Chirka.

\begin{proposition}[Blanchet \cite{Blanchet:1995} and Chirka \cite{Chirka:2003}]
\label{prop:Blanchet}
Let $\Omega$ be a domain in $\mathbb{C}^N$ and $V\subset \Omega$ be a $C^1$-real submanifold with positive real codimension.
Let $u$ be a function of class $C^1$ on $\Omega$.
Then, $u$ is plurisubharmonic on $\Omega$ if so is $u$ on $\Omega-V$.
\end{proposition}

Indeed, Blanchet proved the above removable singularity theorem under an additional condition that $u$ is of class $C^2$ on $\Omega-V$.
However, the condition is eliminated by applying Chirka's theorem \cite{Chirka:2003} with a standard argument with mollifiers. For the completeness, we will confirm Proposition \ref{prop:Blanchet} in Appendix (\S\ref{sec:appendix}).

We return to our setting.
Let $\{\teich_i\}_{i\in I}$ be the stratification 
in Theorem \ref{thm:stratificationofT}.
Let $u$ be a function of class $C^1$ on $\teich_{g,m}-\{x_0\}$
which is bounded above around $x_0$.
Suppose that $u$ is plurisubharmonic on the top stratum $\teich_\infty$.

Let $x_1\in \teich_{g,m}-\{x_0\}$
and $\teich_i$ the stratum containing $x_1$.
Suppose that $u$ is extended as a plurisubharmonic function
on $\teich_j$ for all $j\in I$ with $\dim\teich_j>\dim\teich_i$.
From the locally finiteness of the stratification,
there is a small neighborhood $U$ of $x_1$ such that 
$I(U)=\{j\in I\mid \teich_j\cap U\ne \emptyset\}$ is a finite set
and $\teich_i\cap \overline{\teich_j}\ne\emptyset$
for $j\in I(U)$
from Lemma \ref{lem:locally-finite}.
From the assumption,
$u$ is plurisubharmonic on $U-\teich_i$.
Since $\teich_i$ is a real-analytic submanifold of $\teich_{g,m}$
with positive codimension,
by Proposition \ref{prop:Blanchet},
$u$ is plurisubharmonic on $U$.
This inductive procedure guarantees that $u$ is plurisubharmonic
on $\teich_{g,m}-\{x_0\}$.
Since $u$ is bounded above around $x_0$,
$u$ is extended as a pluriharmonic function on $\teich_{g,m}$
(cf. \cite[Theorem 2.9.22]{Klimek:1991}).
\qed

%


\section{Pluricomplex Green function on the Teichm{\"u}ller space}
\label{sec:pluri-green-function}
In this section,
we will show the following theorem
which implies Theorem \ref{thm:main1},
since the Teichm{\"u}ller distance is the Kobayashi distance on $\teich_{g,m}$
(cf. \cite{Royden:1971} and \S\ref{subsec:complex-analysis}).

\begin{theorem}[Plurisubharmonicity]
\label{thm:plurisubharmonicity}
Let $x_0\in \teich_{g,m}$.
The log-tanh of the Teichm{\"u}ller distance function
\begin{equation*}
\label{eq:log-tanh-Teich}
\teich_{g,m}\ni x\mapsto u_{x_0}(x):=\log\tanh d_T(x_0,x)
\end{equation*}
is plurisubharmonic on $\teich_{g,m}$.
\end{theorem}
Earle \cite{Earle:1977}
showed that $u_{x_0}$ is of class $C^1$ on $\teich_{g,m}-\{x_0\}$.
Since $u_{x_0}(x)\to -\infty$ as $x\to x_0$,
from Theorem \ref{thm:removable-singularities},
it suffices to show that $u_{x_0}$ is plurisubharmonic on 
the top stratum $\teich_\infty$, which will be proved in Theorem \ref{thm:PSH-log-tanh} in \S\ref{subsec:log-tanh-PSH}.


\subsection{Backgrounds from Complex analysis}
\label{subsec:complex-analysis}
Let $X$ be a complex manifold.
Let $p\in X$ and $z=(z_1,\cdots,z_n)$ be a holomorphic local chart
around $p$.
Let $u$ be a $C^2$ function around $p$ on $X$.
For $v=\sum_{i=1}^nv_i(\partial/\partial z_i)\in T_pX$,
we define the \emph{Levi form} of $u$ 
by
\begin{equation*}
\label{eq:Levi-form-definition}
\Levi{u}{v}{\overline{v}}=\sum_{i,j=1}^n\frac{\partial^2 u}{\partial z_i\partial \overline{z}_j}(z(p))
v_i\overline{v_j}.
\end{equation*}
Let $g\colon \{\lambda\in \mathbb{C}\mid |\lambda|<\epsilon\}\to X$
be a holomorphic mapping with $g(0)=p$ and $g_*(\partial/\partial\lambda)=v$.
Then, we see
\begin{equation}
\label{eq:leviform}
\Levi{u}{v}{\overline{v}}=\frac{\partial^2(u\circ g)}{\partial\lambda\partial\overline{\lambda}}(0).
\end{equation}
A $C^2$-function $u$ on $X$ is called \emph{plurisubharmonic}
if $\Levi{u}{v}{\overline{v}}\ge 0$ for $v\in T_pX$ and $p\in X$.
In general,
a function $u$ on a domain $\Omega$ on $\mathbb{C}^N$ is called \emph{plurisubharmonic} if for any $a\in \Omega$ and $b\in \mathbb{C}^N$,
$\lambda\mapsto u(a+\lambda b)$ is subharmonic or 
identically $-\infty$ on every component of $\{\lambda\in \mathbb{C}\mid
a+\lambda b\in \Omega\}$.

A bounded domain $\Omega$ in $\mathbb{C}^N$ is said to be \emph{hyperconvex}
if it admits a negative continuous plurisubharmonic exhaustion
(cf. \cite{Stehle:1974}).
Krushkal \cite{Krushkal:1991} showed that
Teichm{\"u}ller space is hyperconvex (see also \cite{Miyachi:2017}).

Demailly \cite{Demailly:1987} observed that
for any bounded hyperconvex domain $\Omega$ in $\mathbb{C}^n$
and $w\in \Omega$,
there is a unique plurisubharmonic function $g_{\Omega,w}\colon \Omega\to [-\infty,0)$
such that
\begin{enumerate}
\item
$(dd^cg_{\Omega,w})^n=(2\pi)^n\delta_w$,
where $\delta_w$ is the Dirac measure with support at $w$;
and
\item
$g_{\Omega,w}(z)=\sup_v\{v(z)\}$
where the supremum runs over all non-positive plurisubharmonic function $v$
on $\Omega$ with $v(z)\le \log\|z-w\|+O(1)$ around $z=w$.
\end{enumerate}
(cf. \cite[Th\'eor\`eme 4.3]{Demailly:1987}).
The function $g_\Omega(w,z)=g_{\Omega,w}(z)$ is called 
the \emph{pluricomplex Green function} on $\Omega$.
The pluricomplex Green function was introduced by Klimek \cite{Klimek:1985}.
Klimek showed that
\begin{equation}
\label{eq:Klimek-eq}
\log \tanh {\rm Car}_\Omega(z,w)
\le
g_\Omega(z,w)\le \log \tanh {\rm Kob}_\Omega(z,w)
\end{equation}
for $z,w\in \Omega$
and in the second inequality in \eqref{eq:Klimek-eq}, the equality holds if the third term of \eqref{eq:Klimek-eq} is plurisubharmonic,
where ${\rm Car}_\Omega$ and ${\rm Kob}_\Omega$ are the Carath\'eodory distance and the Kobayashi distance on $\Omega$, respectively (cf. \cite[Corollaries 1.2 and 1.4]{Klimek:1985}).

%
%
%
%
%

\subsection{Setting}
\label{subsec:setting-PSH}
Let $q_0\in Z_\infty\cap \mathcal{Q}_{x_0}$
and $x_1=\TeichHomeo_{x_0}(q_0)$.
Let $v\in T_{x_1}\teich_{g,m}$ and $g\colon \{|\lambda|<\epsilon\}\to \teich_{\infty}$ a holomorphic mapping with $g(0)=0$ and
$g_*(\left.\partial/\partial \lambda\right|_{\lambda=0})=v$.
For the simplicity,
let
\begin{align*}
q_\lambda
&=\TeichHomeo_{x_0}^{-1}(g(\lambda))\in Z_\infty\subset \mathcal{Q}_{x_0},
\\
Q_\lambda 
&=\terminaldiff{\tanh^{-1}(\|q_\lambda\|)}{q_\lambda}
\in \mathcal{Q}_{g(\lambda)},
\\
d(\lambda)
&=d_T(x_0,g(\lambda)),\ \mbox{and}
\ d_0 =d(0)=d_T(x_0,x_1)
\end{align*}
for $\lambda\in \{|\lambda|<\epsilon\}$.
Notice again that each $Q_\lambda$ is generic
since the Teichm{\"u}ller mapping preserves the order of singular points.
For calculations later,
we notice from the definition that
\begin{equation}
\label{eq:plurisubharmonicity-13}
\|q_\lambda\|=\tanh (d(\lambda))
\ \mbox{and}
\
\|Q_\lambda\|=e^{-2d(\lambda)}\|q_\lambda\|.
\end{equation}

\subsubsection{}
We define $\cohomologyclass{v}_1$, $\cohomologyclass{v}_2\in
\homorel(q_0)$
($\cong T_{q_0}\mathcal{Q}_{g,m}$)
by
\begin{equation}
\label{eq:plurisubharmonicity-2}
\cohomologyclass{u}[q_\lambda]=\cohomologyclass{u}[q_0]
+\lambda \cohomologyclass{v}_1+\overline{\lambda}
\cohomologyclass{v}_2+o(|\lambda|)
\end{equation}
as $\lambda\to 0$
on $H_1(\tilde{M}_{q_0},\mathbb{R})^-\cong
H_1(\tilde{M}_{q_0},\tilde{\Sigma}_{ub}(q_0),\mathbb{R})^-$.
Since $q_\lambda\in \mathcal{Q}_{x_0}$ for all $\lambda$,
we deduce $\tvector{\cohomologyclass{v}_i}{q_0}=0$ for $i=1,2$
and $\cohomologyclass{v}_1$,
$\cohomologyclass{v}_2\in \homorel_0(q_0)$
($\cong \mathcal{Q}_{x_0}\subset T_{q_0}\mathcal{Q}_{g,m}$).
We will use the notation
\begin{equation}
\label{eq:derivative-v}
\deriv{x_1}{v}=\cohomologyclass{v}_1, \ \mbox{and} \
\derivbar{x_1}{v}=\cohomologyclass{v}_2
\end{equation}
after calculating the first derivative and the Levi form of the Teichm{\"u}ller distance
(cf. \S\ref{subsec:Complex-tangent-space}).
However, 
in the following calculation,
we will use the notation $\cohomologyclass{v}_1$
and $\cohomologyclass{v}_2$ for the simplicity.

From \eqref{eq:teichmuller-deformation-homomorphism1}
and \eqref{eq:plurisubharmonicity-2},
\begin{align*}
\cohomologyclass{u}[Q_\lambda]
&=
{\rm Re}(\cohomologyclass{u}[q_\lambda])
+\sqrt{-1}e^{-2d(\lambda)}
{\rm Im}(\cohomologyclass{u}[q_\lambda])
\\
&=\cohomologyclass{u}[q_0]+
\lambda
\cohomologyclass{w}_1
+
\overline{\lambda}\cohomologyclass{w}_2+o(|\lambda|)
\nonumber
\end{align*}
on $H_1(\tilde{M}_{q_0},\mathbb{R})^-$
as $\lambda\to 0$,
where
$\cohomologyclass{w}_1$ and $\cohomologyclass{w}_2$
are in $\homorel(q_0)$ defined by
\begin{align}
\begin{cases}
\cohomologyclass{w}_1
&=
{\displaystyle
\frac{1+e^{-2d_0}}{2}\cohomologyclass{v}_1+
\frac{1-e^{-2d_0}}{2}\overline{\cohomologyclass{v}_2}
-d_\lambda\,
e^{-2d_0}(\cohomologyclass{u}[q_0]-\overline{\cohomologyclass{u}[q_0]})
}
\\
\cohomologyclass{w}_2
&=
{\displaystyle
\frac{1-e^{-2d_0}}{2}\overline{\cohomologyclass{v}_1}+
\frac{1+e^{-2d_0}}{2}\cohomologyclass{v}_2
-\overline{d_{\lambda}}\,
e^{-2d_0}(\cohomologyclass{u}[q_0]-\overline{\cohomologyclass{u}[q_0]}),
}
\end{cases}
\label{eq:plurisubharmonicity-5}
\end{align}
where $d_\lambda$ is the $\lambda$-derivative of the Teichm{\"u}ller distance function $d(\lambda)$
at $\lambda=0$.
In
\eqref{eq:plurisubharmonicity-5},
we canonically identify $\homorel(Q_0)$ with $\homorel(q_0)$,
and $\cohomologyclass{w}_1$ and $\cohomologyclass{w}_2$ stands for
tangent vectors in $T_{Q_0}\mathcal{Q}_{g,m}\cong \homorel(Q_0)$,
while the right-hand sides of
\eqref{eq:plurisubharmonicity-5}
are tangent vectors in $\homorel(q_0)\cong T_{q_0}\mathcal{Q}_{g,m}$.
See the discussion in the proof of Theorem \ref{thm:stratificationofT}
and (2) of Remark \ref{remark:1}.

Since $\projQ(\cohomologyclass{u}[Q_\lambda])=g(\lambda)$,
we have
\begin{equation}
\label{eq:plurisubharmonicity-6}
\tvector{\cohomologyclass{w}_1}{Q_0}=v \quad
\mbox{and}\quad
\tvector{\cohomologyclass{w}_2}{Q_0}=0
\end{equation}
and $\cohomologyclass{w}_2\in \homorel_0(Q_0)$,
since $\projQ$ is holomorphic.

\subsubsection{}
Since $H_1(\tilde{M}_{q_0},\mathbb{R})^-\cong
H_1(\tilde{M}_{q_0},\tilde{\Sigma}_{ub}(q_0),\mathbb{R})^-$,
$\homorel(q_0)$ is canonically isomorphic to the cohomology group
$H^1(\tilde{M}_{q_0},\mathbb{R})^-$.
We define the \emph{wedge product} $\wedge$ on $\homorel(q_0)$
by
\begin{equation*}
\label{eq:plurisubharmonicity-7}
\cohomologyclass{x}\wedge \cohomologyclass{y}
=
\int_{\tilde{M}_{q_0}}
\frac{\pi_{q_0}^*(\eta_{\tvector{\overline{\cohomologyclass{x}}}{q_0}})}{\omega_{q_0}}
\wedge
\overline{
\left(
\frac{\pi_{q_0}^*(\eta_{\tvector{\cohomologyclass{y}}{q_0}})}{\omega_{q_0}}
\right)}
-
\frac{\pi_{q_0}^*(\eta_{\tvector{\overline{\cohomologyclass{y}}}{q_0}})}{\omega_{q_0}}
\wedge
\overline{
\left(
\frac{\pi_{q_0}^*(\eta_{\tvector{\cohomologyclass{x}}{q_0}})}{\omega_{q_0}}
\right)}
\end{equation*}
for $\cohomologyclass{x}$, $\cohomologyclass{y}\in \homorel(q_0)$.
For generic $q\in \mathcal{Q}_{g,m}$,
\begin{align}
\sqrt{-1}\cohomologyclass{u}[q]\wedge \overline{\cohomologyclass{u}[q]}
&=4\|q\|
\label{eq:plurisubharmonicity-10}
\end{align}
(cf. \eqref{eq:Teichmuller-homeo-distance} and Example \ref{example:generic-eta}).

\begin{example}[Teichm{\"u}ller disk]
\label{example:teichmuller-disk}
The \emph{Teichm{\"u}ller disk} associated to $q_0$ is defined as an isometric holomorphic disk
in $\teich_{g,m}$
defined by the holomorphic family of Beltrami differentials
$$
\mathbb{D}\ni \lambda \mapsto 
\left(\frac{\lambda+\tanh(d_0)}{1+\tanh(d_0)\lambda}\right)
\frac{\overline{q_0}}{|q_0|}
$$
on $M_0$.
With the Teichm{\"u}ller homeomorphism \eqref{eq:Teichmuller-homeomorphism},
the Teichm{\"u}ller disk is described as
\begin{equation}
\label{eq:Teichmuller-disk-qd}
\mathbb{D}\ni \lambda\mapsto \TeichHomeo_{x_0}
\left(
\left(\frac{\overline{\lambda}+\tanh(d_0)}{1+\tanh(d_0)\overline{\lambda}}\right)\frac{q_0}{\|q_0\|}\right).
\end{equation}
For $\lambda\in \mathbb{D}$.
let $\teichdisk_{q_0} (\lambda)$ be the right-hand side of \eqref{eq:Teichmuller-disk-qd}.
By definition,
$\teichdisk_{q_0} (0)=x_1$ and $\teichdisk_{q_0} (-\tanh(d_0))=x_0$.
Then
\begin{align}
\cohomologyclass{u}[\TeichHomeo_{x_0}^{-1}(\teichdisk_{q_0} (\lambda))]
&=\frac{1}{\|q_0\|^{1/2}}\left(\frac{\overline{\lambda}+\tanh(d_0)}{1+\tanh(d_0)\overline{\lambda}}\right)^{1/2}\cohomologyclass{u}[q_0]
\label{eq:rep-q0-1}
\\
&=\cohomologyclass{u}[q_0]+\frac{\overline{\lambda}}{\sinh(2d_0)}\cohomologyclass{u}[q_0]+o(|\lambda|),
\nonumber
\end{align}
where the branch of the square root taken to be $1^{1/2}=1$.
\end{example}

\subsection{The first variation of the Teichm{\"u}ller distance}
We give the first variational formula of the Teichm{\"u}ller distance function
in our setting.
From Riemann's formula,
\eqref{eq:plurisubharmonicity-13} and \eqref{eq:plurisubharmonicity-10},
we deduce
\begin{align}
d_\lambda
&=\left.
(\tanh^{-1}(\|q_\lambda\|))_{\lambda}
\right|_{\lambda=0}
=
\frac{1}{1-\|q_0\|^2}\frac{\sqrt{-1}}{4}
(\cohomologyclass{v}_1\wedge \overline{u[q_0]}
+u[q_0]\wedge \overline{\cohomologyclass{v}_2})
\nonumber
\\
&=\frac{\sqrt{-1}\cosh^2(d_0)}{4}
(\cohomologyclass{v}_1\wedge \overline{u[q_0]}
+
u[q_0]\wedge \overline{\cohomologyclass{v}_2}).
\label{eq:plurisubharmonicity-14} 
\end{align}
On the other hand,
Earle \cite{Earle:1977} gave the first variational formula
\begin{equation*}
\label{eq:Earle-formula1}
d_\lambda
=\frac{1}{2\|Q_0\|}\int_{M_1}\mu Q_0
\end{equation*}
where $\mu$ is the infinitesimal Beltrami differential on $M_1$ representing $v$.
From \eqref{eq:plurisubharmonicity-13},
\eqref{eq:plurisubharmonicity-6}
and \eqref{eq:plurisubharmonicity-10},
\begin{align*}
&d_\lambda=\frac{1}{2\|Q_0\|}\int_{M_1}\mu Q_0
=\frac{1}{2e^{-2d_0}\tanh(d_0)}
\int_{M_1}\frac{\overline{\eta_v}}{|Q_0|}Q_0
\\
&=\frac{e^{2d_0}}{2\tanh(d_0)}
\frac{-\sqrt{-1}}{4}
\int_{\tilde{M}_{Q_0}}
\overline{
\left(
\frac{
\pi_{Q_0}^*(\eta_{\tvector{\cohomologyclass{w}_1}{Q_0}})
}
{\omega_{Q_0}}
\right)
}\wedge \omega_{Q_0}
=
\frac{-\sqrt{-1}e^{2d_0}}{8\tanh(d_0)}
\cohomologyclass{w}_1\wedge \cohomologyclass{u}[Q_0]
\\
&=
\frac{-\sqrt{-1}e^{2d_0}}{8\tanh(d_0)}
\left(
\frac{1+e^{-2d_0}}{2}\cohomologyclass{v}_1+
\frac{1-e^{-2d_0}}{2}\overline{\cohomologyclass{v}_2}
-d_\lambda\,
e^{-2d_0}(\cohomologyclass{u}[q_0]-\overline{\cohomologyclass{u}[q_0]})
\right)
\\
&\qquad\qquad\qquad\qquad \wedge
\left(\frac{1+e^{-2d_0}}{2}\cohomologyclass{u}[q_0]+
\frac{1-e^{-2d_0}}{2}\overline{\cohomologyclass{u}[q_0]}
\right)
\\
&=
\frac{-\sqrt{-1}e^{2d_0}}{8\tanh(d_0)}
\left(
\frac{1-e^{-4d_0}}{4}\overline{\cohomologyclass{v}_2}\wedge \cohomologyclass{u}[q_0]+
d_\lambda e^{-2d_0}\frac{1+e^{-2d_0}}{2}\overline{\cohomologyclass{u}[q_0]}\wedge \cohomologyclass{u}[q_0]
\right.
\\
&\quad \qquad\qquad\qquad\qquad+
\left.
\frac{1-e^{-4d_0}}{4}\cohomologyclass{v}_1\wedge \overline{\cohomologyclass{u}[q_0]}
-
d_\lambda e^{-2d_0}\frac{1-e^{-2d_0}}{2}\cohomologyclass{u}[q_0]\wedge \overline{\cohomologyclass{u}[q_0]}
\right)
\\
&=
-\frac{\sqrt{-1}\cosh^2(d_0)}{8}
\left(
\cohomologyclass{v}_1\wedge \overline{\cohomologyclass{u}[q_0]}
-\cohomologyclass{u}[q_0]\wedge \overline{\cohomologyclass{v}_2}
\right)
+
\frac{\sqrt{-1}d_\lambda }{8\tanh(d_0)}
\cohomologyclass{u}[q_0]\wedge \overline{\cohomologyclass{u}[q_0]}
\\
&=-\frac{\sqrt{-1}\cosh^2(d_0)}{8}
\left(
\cohomologyclass{v}_1\wedge \overline{\cohomologyclass{u}[q_0]}
-\cohomologyclass{u}[q_0]\wedge \overline{\cohomologyclass{v}_2}
\right)
+\frac{d_\lambda}{2}.
\end{align*}
Therefore,
we obtain
\begin{equation}
\label{eq:Earle-formula2}
d_\lambda=
\frac{\sqrt{-1}\cosh^2(d_0)}{4}
\left(
-\cohomologyclass{v}_1\wedge \overline{\cohomologyclass{u}[q_0]}
+\cohomologyclass{u}[q_0]\wedge \overline{\cohomologyclass{v}_2}
\right).
\end{equation}
Thus,
from \eqref{eq:plurisubharmonicity-14}  and \eqref{eq:Earle-formula2}
we conclude the following.
\begin{lemma}[First variational formula]
\label{lem:first-variational-formula}
Under the notations in \S\ref{subsec:setting-PSH},
we have
\begin{equation*}
\label{eq:first-variational-formula_our-setting}
d_\lambda
=
\frac{\sqrt{-1}\cosh^2(d_0)}{4}
\cohomologyclass{u}[q_0]\wedge \overline{\cohomologyclass{v}_2}
\end{equation*}
and $\cohomologyclass{v}_1\wedge \overline{\cohomologyclass{u}[q_0]}=0$.
\end{lemma}
%

\subsection{Levi form of $d_T$}
Let $u[q_\lambda]_\lambda\in \homorel(q_0)$ be the $\lambda$-derivative of
the family $\{\cohomologyclass{u}[q_\lambda]\}_{|\lambda|<\epsilon}$
of the representation.
Namely,
$u[q_\lambda]_\lambda(C)=(u[q_\lambda](C))_\lambda$
for $C\in H_1(\tilde{M}_{q_0},\mathbb{R})^-$.
Notice from the notation in \S\ref{subsec:setting-PSH}
that $u[q_\lambda]_\lambda=\cohomologyclass{v}_1$
at $\lambda=0$.
We also define $u[q_\lambda]_{\overline{\lambda}}$
and $u[q_\lambda]_{\lambda\overline{\lambda}}$ in the same manner.
From 
Lemma \ref{lem:first-variational-formula},
the $\lambda$-derivative of $d(\lambda)=d_T(x_0,g(\lambda))$ 
on a disk $\{|\lambda|<\epsilon\}$ is 
rewritten as
$$
d_\lambda(\lambda)=\frac{\sqrt{-1}\cosh^2(d(\lambda))}{4}
\cohomologyclass{u}[q_\lambda]\wedge \overline{\cohomologyclass{u}[q_\lambda]_{\overline{\lambda}}}.
$$
Therefore,
\begin{align*}
d_{\lambda\overline{\lambda}}(0)
&=\frac{\sqrt{-1}}{4}\cdot 2\sinh(d_0)\cosh(d_0)d_{\overline{\lambda}}\cdot
\cohomologyclass{u}[q_0]\wedge \overline{\cohomologyclass{v}_2}
\\
&\quad
+\frac{\sqrt{-1}\cosh^2(d_0)}{4}
\cohomologyclass{v}_2\wedge \overline{\cohomologyclass{v}_2}
+\frac{\sqrt{-1}\cosh^2(d_0)}{4}\cohomologyclass{u}[q_0]\wedge
\overline{\cohomologyclass{u}[q_0]_{\lambda\overline{\lambda}}\mid_{\lambda=0}}.
\end{align*}
From Lemma \ref{lem:first-variational-formula} again,
\begin{align*}
0
&=\left(
\cohomologyclass{u}[q_\lambda]\wedge \overline{\cohomologyclass{u}[q_\lambda]_\lambda}
\right)_\lambda
=
\cohomologyclass{u}[q_\lambda]_\lambda\wedge \overline{\cohomologyclass{u}[q_\lambda]_\lambda}
+
\cohomologyclass{u}[q_\lambda]\wedge
\overline{\cohomologyclass{u}[q_\lambda]_{\lambda\overline{\lambda}}}
\end{align*}
Thus,
the Laplacian $d_{\lambda\overline{\lambda}}(0)$ of the distance function
$d(\lambda)$ at $\lambda=0$ is
\begin{equation}
\label{eq:Leviform-Teichmullerdist}
\frac{\cosh^3(d_0)\sinh(d_0)}{8}
|\cohomologyclass{u}[q_0]\wedge \overline{\cohomologyclass{v}_2}|^2
+\frac{\sqrt{-1}\cosh^2(d_0)}{4}(
\cohomologyclass{v}_2\wedge \overline{\cohomologyclass{v}_2}
-
\cohomologyclass{v}_1\wedge \overline{\cohomologyclass{v}_1}
).
\end{equation}

\subsection{Complex tangent spaces of the spheres}
\label{subsec:Complex-tangent-space}
We use the notation \eqref{eq:derivative-v}.
Notice that
\begin{align*}
&T_{x_1}\teich_{g,m}\ni v\mapsto \deriv{x_1}{v} \in \homorel_0(q_0) \\
&T_{x_1}\teich_{g,m}\ni v\mapsto \derivbar{x_1}{v} \in \homorel_0(q_0)
\end{align*}
are complex and anti-complex linear respectively.

From \eqref{eq:leviform}, Lemma \ref{lem:first-variational-formula} and
\eqref{eq:Leviform-Teichmullerdist},
the first derivative and the Levi form of the Teichm{\"u}ller distance
function $\teich_{g,m}\ni x\mapsto d_T(x_0,x)$
at $x_1\in \teich_{g,m}-\{x_0\}$
are rewritten as
\begin{align}
\partial\,d_T(x_0,\,\cdot\,)[v]
&=\frac{\sqrt{-1}\cosh^2(d_0)}{4}
\cohomologyclass{u}[q_0]\wedge \overline{\derivbar{x_1}{v}}
\nonumber
\\
\Levi{d_T(x_0,\,\cdot\,)}{v}{\overline{v}}
&=
\frac{\cosh^3(d_0)\sinh(d_0)}{8}
|\cohomologyclass{u}[q_0]\wedge \overline{\derivbar{x_1}{v}}|^2
\label{eq:Leviform-v}
\\
&+\frac{\sqrt{-1}\cosh^2(d_0)}{4}(
\derivbar{x_1}{v}\wedge \overline{\derivbar{x_1}{v}}
-
\deriv{x_1}{v}\wedge \overline{\deriv{x_1}{v}}
)
\nonumber
\end{align}
for $v\in T_{x_1}\teich_{g,m}$.
For $v_1,v_2\in T_{x_1}\teich_{g,m}$,
the Hermitian form of the Levi form 
is represented as
\begin{align}
\Levi{d_T(x_0,\,\cdot\,)}{v_1}{\overline{v_2}}&=
\frac{\cosh^3(d_0)\sinh(d_0)}{8}
\cohomologyclass{u}[q_0]\wedge \overline{\derivbar{x_1}{v_1}}
\cdot \overline{\cohomologyclass{u}[q_0]}\wedge \derivbar{x_1}{v_2}
\label{eq:Leviform-v2}
\\
&+\frac{\sqrt{-1}\cosh^2(d_0)}{4}(
\derivbar{x_1}{v_2}\wedge \overline{\derivbar{x_1}{v_1}}
-
\deriv{x_1}{v_1}\wedge \overline{\deriv{x_1}{v_2}}
)
\nonumber
\end{align}
from \eqref{eq:Leviform-v}.

For $r>0$,
we consider the sphere $S(x_0,r)=\{x\in \teich_{g,m}\mid d_T(x_0,x)=r\}$
of the Teichm{\"u}ller distance.
For $x_1\in S(x_0,r)$,
we define
\begin{align*}
H^{1,0}_{x_1}=H^{1,0}_{x_1}(S(x_0,r))
&=\{v\in T_{x_1}\teich_{g,m}\mid \partial\,d_T(x_0,\,\cdot\,)[v]=0\}
\\
&
=\{v\in T_{x_1}\teich_{g,m}\mid \cohomologyclass{u}[q_0]\wedge \overline{\derivbar{x_1}{v}}=0\}.
\end{align*}
The subspace $H^{1,0}_{x_1}$ is
the \emph{$(1,0)$-part of the complex tangent space of the sphere $S(x_0,r)$}
(cf. \cite[\S7]{Boggess:1991}).
Let $\normalvec{x_1}\in T_{x_1}\teich_{g,m}$ be the tangent vector
associated to the infinitesimal Beltrami differential $\overline{Q_0}/|Q_0|$
where $Q_0$ is the terminal differential of the Teichm{\"u}ller mapping
from $x_0$ to $x_1$.
The vector $\normalvec{x_1}\in T_{x_1}\teich_{g,m}$ is tangent 
to the Teichm{\"u}ller disk passing $x_0$ and $x_1$.
From \eqref{eq:rep-q0-1},
\begin{align}
&\deriv{x_1}{\normalvec{x_1}}=0,\quad
\derivbar{x_1}{\normalvec{x_1}}=
\frac{1}{\sinh(2d_0)}\cohomologyclass{u}[q_0].
\label{eq:derivative-normal-vector}
\end{align}
%

\begin{proposition}[CR tangent space]
\label{prop:CR}
The complex tangent space $H^{1,0}_{x_1}$ is perpendicular to
$\normalvec{x_1}$ with respect to the Levi form of
the Teichm{\"u}ller distance.
\end{proposition}

\begin{proof}
From \eqref{eq:Leviform-v2} and \eqref{eq:derivative-normal-vector},
$\Levi{d_T(x_0,\,\cdot\,)}{v}{\overline{\normalvec{x_1}}}=0$
for $v\in H^{1,0}_{x_1}$.
\end{proof}

\begin{lemma}[Non-negativity on $H^{1,0}$]
\label{lem:non-negativity}
For $v\in H^{1,0}_{x_1}$,
$$
\sqrt{-1}(\derivbar{x_1}{v}\wedge \overline{\derivbar{x_1}{v}}
-
\deriv{x_1}{v}\wedge \overline{\deriv{x_1}{v}})\ge 0
$$
\end{lemma}

\begin{proof}
%
%
%
%
From \cite[Corollary 1.2]{Miyachi:2017},
the Teichm{\"u}ller distance function $\teich_{g,m}\ni x\mapsto
d_T(x_0,x)$ is plurisubharmonic
(see also \cite[Corollary 3]{Krushkal:1992}).
Hence the Levi-form of the distance function is non-negative
on $H^{1,0}_{x_1}$.
The assertion follows from \eqref{eq:Leviform-v}.
\end{proof}

\subsection{Log-tanh of $d_T$ is plurisubharmonic on $\teich_\infty$}
\label{subsec:log-tanh-PSH}
We set
$$
u_{x_0}(x)=\log\tanh(d_T(x_0,x))
$$
for $x\in \teich_{g,m}$.
From \eqref{eq:Teichmuller-homeo-distance},
$$
u_{x_0}(\TeichHomeo_{x_0}(q))=\log\|q\|
$$
for $q\in \mathcal{UQ}_{x_0}$.
Recall that $\teich_\infty$ 
is the top stratum of the stratification of $\teich_{g,m}-\{x_0\}$
which is obtained in Theorem \ref{thm:stratificationofT}.

\begin{theorem}[Plurisubharmonicity]
\label{thm:PSH-log-tanh}
$u_{x_0}$ is plurisubharmonic on $\teich_\infty$.
\end{theorem}

\begin{proof}
Let $x_1\in \teich_\infty$ and $q_0\in \teich_{x_0}$ with $x_1=\TeichHomeo_{x_0}(q_0)$ as above.
Since $u_{x_0}(g(\lambda))=\log \|q_\lambda\|$
for $\lambda\in \{|\lambda|<\epsilon\}$,
from the direct calculation,
the Levi form of $u_{x_0}$
is given as
\begin{align}
\label{eq:Levi-Green}
\Levi{u_{x_0}}{v}{\overline{v}}
&=\frac{\sqrt{-1}}{4\tanh(d_0)}\left(
\derivbar{x_1}{v}\wedge \overline{\derivbar{x_1}{v}}
-
\deriv{x_1}{v}\wedge \overline{\deriv{x_1}{v}}
\right)
\\
&\qquad\qquad
-
\frac{1}{16\tanh^2(d_0)}
|\cohomologyclass{u}[q_0]\wedge \overline{\derivbar{x_1}{v}}|^2
\nonumber
\end{align}
for $v\in T_{x_1}\teich_{g,m}$.
From
\eqref{eq:plurisubharmonicity-13} and \eqref{eq:plurisubharmonicity-10},
\eqref{eq:Leviform-v}
and \eqref{eq:derivative-normal-vector},
\begin{align*}
&\Levi{u_{x_0}}{\normalvec{x_1}}{\overline{\normalvec{x_1}}}
\\
&=\frac{\sqrt{-1}}{4\tanh(d_0)}
\frac{\cohomologyclass{u}[q_0]}{\sinh(2d_0)}\wedge 
\overline{\frac{\cohomologyclass{u}[q_0]}{\sinh(2d_0)}}
-
\frac{1}{16\tanh^2(d_0)}
\left|
\cohomologyclass{u}[q_0]\wedge 
\overline{\frac{\cohomologyclass{u}[q_0]}{\sinh(2d_0)}}
\right|^2
=0.
\end{align*}
From Lemma \ref{lem:non-negativity},
the Levi form of $u_{x_0}$ is non-negative on $H^{1,0}_{x_1}$.
Applying the calculation in Proposition \ref{prop:CR}
to \eqref{eq:Levi-Green},
we also deduce that
the normal vector $\normalvec{x_1}$ is perpendicular
to $H^{1,0}_{x_1}$ with respect 
to the Levi form of $u_{x_0}$.
Therefore,
the Levi form of $u_{x_0}$ is non-negative
on the whole $T_{x_1}\teich_{g,m}$.
\end{proof}

\subsection{Proof of Corollary \ref{coro:Demailly-distance}}
\label{subsec:proof-corollary}
As mentioned in Introduction, we identify $\teich_{g,m}$ with the Bers slice with base point $x_0\in \teich_{g,m}$. Then, $\teich_{g,m}$ is a hyperconvex domain. Notice that for $x,y\in \teich_{g,m}$,
\begin{align*}
\frac{g_{\teich_{g,m}}(x,z)}{g_{\teich_{g,m}}(y,z)}
&
=e^{-2(d_{T}(x,z)-d_{T}(y,z))}(1+o(1))
\le e^{2d_{T}(x,y)}(1+o(1))
\end{align*}
when $z\to \partial \teich_{g,m}$.
Therefore, $\boldsymbol{\delta}_{\teich_{g,m}}(x,y)  \le 2d_{T}(x,y)$.

On the other hand,
when we consider a divergent sequence $\{x_{n}\}_{n}$ in $\teich_{g,m}$
along the Teichm{\"u}ller geodesic connecting $x$ and $y$
with $d_{T}(x,x_{n})<d_{T}(y,x_{n})$,
we have
\begin{align*}
\frac{g_{\teich_{g,m}}(x,x_{n})}{g_{\teich_{g,m}}(y,x_{n})}
&
=e^{-2(d_{T}(x,x_{n})-d_{T}(y,x_{n}))}(1+o(1))
=e^{2d_{T}(x,y)}(1+o(1))
\end{align*}
as $n\to \infty$,
and hence $2d_{T}(x,y)\le \boldsymbol{\delta}_{\teich_{g,m}}(x,y)$.

\section{Topological description of the Levi form}
\label{sec:Topological-description}
The space $\mathcal{MF}$ carries a natural symplectic structure
with the \emph{Thurston symplectic form}  $\omega_{Th}$ (cf. \cite{Penner_Harer:1992}).
Dumas \cite[Theorem 5.3]{Dumas:2015} introduced a K\"ahler (symplectic)
structure on each stratum of $\mathcal{Q}_{x_0}$
discussed in \S\ref{subsec:stratification-Qy0}
which defined from the Levi-form of the $L^1$-norm on $\mathcal{Q}_{x_0}$,
and observed that the Hubbard-Masur homeomorphism
\eqref{eq:HM-symp}
is a real-analytic symplectomorphism on each stratum of $\mathcal{Q}_{x_0}$
(cf. \cite[Theorem 5.8]{Dumas:2015}).
In fact,
when $q_0\in\mathcal{Q}_{x_0}$ is generic,
Dumas showed that
the Hubbard-Masur homeomorphism
\eqref{eq:HM-symp} is a diffeomorphism around $q_0$ and satisfies
\begin{align}
\label{eq:Dumas-symplectic-form}
\omega_{Th}(d\mathcal{V}_{x_0}(\psi_1),d\mathcal{V}_{x_0}(\psi_2))
&=
{\rm Im}\int_{M_0}\frac{\psi_1\overline{\psi_2}}{4|q_0|}
\\
&
=\frac{1}{8}\int_{\tilde{M}_{q_0}}
{\rm Re}
\left(
\frac{\pi_{q_0}^*(\psi_1)}{\omega_{q_0}}
\right)
\wedge
{\rm Re}\left(
\frac{\pi_{q_0}^*(\psi_2)}{\omega_{q_0}}
\right)
\nonumber
\end{align}
for $\psi_1,\psi_2\in T_{q_0}\mathcal{Q}_{x_0}=\mathcal{Q}_{x_0}$
(cf. \eqref{eq:Hermite-form-cover} and \cite[\S 5.2]{Dumas:2015}).
\begin{remark}
\label{remark:HM-homeo}
Dumas \cite{Dumas:2015} discussed the Hubbard-Masur homeomorphism
\eqref{eq:HM-symp} by assigning the horizontal foliations to quadratic differentials
in accordance with Hubbard and Masur's original discussion.
The original Hubbard-Masur homeomorphism
$\mathcal{H}_{x_0}\colon \mathcal{Q}_{x_0}\to \mathcal{MF}$
satisfies $\mathcal{V}_{x_0}(q)=\mathcal{H}_{x_0}(-q)$.
Hence,
$d\mathcal{V}_{x_0}(\psi)=-d\mathcal{H}_{x_0}(\psi)$
for $\psi\in \mathcal{Q}_{x_0}=T_{q_0}\mathcal{Q}_{x_0}$.
Thus,
the formula \eqref{eq:Dumas-symplectic-form} also holds in our case.
\end{remark}

Let us go back to the notion in \S\ref{subsec:setting-PSH}.
Notice that
\begin{align}
\label{eq:Levi-1}
\cohomologyclass{v}_2\wedge \overline{\cohomologyclass{v}_2}
-
\cohomologyclass{v}_1\wedge \overline{\cohomologyclass{v}_1}
&=
(\cohomologyclass{v}_2+\overline{\cohomologyclass{v}}_1)\wedge
(\overline{\cohomologyclass{v}}_2+\cohomologyclass{v}_1)
=\left(
{\rm Re}(\cohomologyclass{u}[q_\lambda])
\right)_{\overline{\lambda}}
\wedge
\overline{
\left(
{\rm Re}(\cohomologyclass{u}[q_\lambda])
\right)_{\overline{\lambda}}}
\\
&
=-\frac{\sqrt{-1}}{2}
{\rm Re}
\left(
\cohomologyclass{u}[q_\lambda]_{\xi_1}
\right)
\wedge
{\rm Re}
\left(
\cohomologyclass{u}[q_\lambda]_{\xi_2}
\right)
\nonumber
\end{align}
at $\lambda=0$,
where $\lambda=\xi_1+\xi_2\sqrt{-1}$.
Therefore,
when
$$
q_\lambda=\TeichHomeo_{x_0}(g(\lambda))
=q_0+\xi_1\psi_1+\xi_2\psi_2+o(|\lambda|)
$$
as $\lambda=\xi_1+\xi_2\sqrt{-1}\to 0$,
from \eqref{eq:plurisubharmonicity-13},
\eqref{eq:plurisubharmonicity-10},
\eqref{eq:Dumas-symplectic-form} and \eqref{eq:Levi-1},
we deduce the \emph{topological description} of the Levi form of the pluricomplex Green function $u_{x_0}$ on $\teich_{g,m}$:
$$
\Levi{u_{x_0}}{v}{\overline{v}}
=\frac{1}{\|q_0\|}
\omega_{Th}(d\mathcal{V}_{x_0}(\psi_1),d\mathcal{V}_{x_0}(\psi_2))
-
\frac{|\cohomologyclass{u}[q_0]\wedge \overline{\derivbar{x_1}{v}}|^2}{16\|q_0\|^2},
$$
where
\begin{align*}
\cohomologyclass{u}[q_0]\wedge \overline{\derivbar{x_1}{v}}
&=8\left(\omega_{Th}(d\mathcal{V}_{x_0}(q_0),d\mathcal{V}_{x_0}(\psi_2-\sqrt{-1}\psi_1))\right.
\\
&\qquad \qquad
+\left.\sqrt{-1}\,\omega_{Th}(d\mathcal{V}_{x_0}(q_0),d\mathcal{V}_{x_0}(\psi_1+\sqrt{-1}\psi_2))\right).
\end{align*}
In particular,
when $v\in H^{1,0}_{x_1}$ for $x_1=g(0)$,
we conclude
\begin{equation*}
\label{eq:levi-3}
\Levi{u_{x_0}}{v}{\overline{v}}
=\frac{1}{\|q_0\|}
\omega_{Th}(d\mathcal{V}_{x_0}(\psi_1),d\mathcal{V}_{x_0}(\psi_2)).
\end{equation*}
As a corollary,
we deduce
\begin{equation}
\label{eq:levi-4}
\omega_{Th}(d\mathcal{V}_{x_0}(\psi_1),d\mathcal{V}_{x_0}(\psi_2))
\ge 
\frac{|\cohomologyclass{u}[q_0]\wedge \overline{\derivbar{x_1}{v}}|^2}{16\|q_0\|}
\ge 0.
\end{equation}
From the definition,
$d\mathcal{V}_{x_0}(\psi_i)\in T_{\verticalfoliation{q_0}}\mathcal{MF}$
is the infinitesimal transverse cocycle (in the sense of Bonahon \cite{Bonahon:1996})
of the initial differentials $\{q_\lambda\}_{\lambda}$
associated along the $\xi_i$-direction at $\lambda=0$.
Thus,
the non-negativity derived from \eqref{eq:levi-4}
of the Thurston symplectic pairing
between the infinitesimal transverse cocycles $d\mathcal{V}_{x_0}(\psi_1)$
and $d\mathcal{V}_{x_0}(\psi_2)$ is a necessary condition
to describe complex-analytic deformations of Teichm{\"u}ller 
mappings
from the topological aspect in Teichm{\"u}ller theory.

\section{Appendix}
\label{sec:appendix}
In this section, we shall check Proposition \ref{prop:Blanchet}.
As noticed before, Proposition \ref{prop:Blanchet} is proved with a condition that $u$ is of $C^2$ by Blanchet \cite{Blanchet:1995}.
Chirka \cite{Chirka:2003} extends Blanchet's result as follows: When $\Omega$ and $V$ are taken as Proposition \ref{prop:Blanchet}, a subharmonic function on $\Omega$ which is plurisubharmonic on $\Omega\setminus V$ is plurisubharmonic on $\Omega$ without any smoothness condition. Hence, Proposition \ref{prop:Blanchet} follows from the following proposition which will be well-known. However, the author can not find any suitable reference and gives a proof for the completeness.

\begin{proposition}
\label{prop:subharmonicitiy}
Let $\Omega$ be a domain in $\mathbb{R}^N$ and $V$ be a $C^1$-submanifold of $\Omega$ with positive codimension.
When a $C^1$-function $u$ on $\Omega$ is subharmonic on $\Omega\setminus V$, $u$ is subharmonic on $\Omega$.
\end{proposition}

\begin{proof}
Since the subharmonicity is a local condition, we may assume that $\Omega$ is a ball centered at $z_0$ and $V$ is a hypersurface (\cite[Theorem 2.4.1]{Klimek:1991}). Let $z_0\in V$ and ${\bf n}$ be the normal vector to $V$ at $z_0$.  By applying the inverse mapping theorem, we may also assume that 
\begin{itemize}
\item for small $\delta_0>0$, $V_t=V+t{\bf n}$ ($|t|<\delta_0$) separates $\Omega$ into two domains, and intersects $\partial \Omega$ transversely;
\item $V_t\cap V_{t'}=\emptyset$ ($t\ne t'$); and
\item for any $0<\delta<\delta_0$, the union $N(\delta)=\cup_{|t|<\delta}V_t$ forms an open set. 
\end{itemize}
For a domain $E$ and a constant $r>0$ we set $E_r=\{x\in E\mid {\rm dist}(x,\partial E)>r\}$.
Let $\Omega_\delta\setminus N(\delta)=\Omega^+_\delta\sqcup\Omega^-_\delta$ and $F^s_\delta=\overline{\Omega^s_\delta}\cap V_{(-1)^s \delta}$ for $s=\pm$ (we assume $F^\pm_\delta\ne \emptyset$). Then, $\partial \Omega^s_\delta=(\partial \Omega_\delta\cap \overline{\Omega^s_\delta})\cup F^s_\delta$.

Let $\epsilon>0$. Let $u_\epsilon=\chi_\epsilon*u$ be the convolution where $\chi_\epsilon$ is a standard smoothing kernel (mollifier) supported on the ball centered at the origin with radius $\epsilon$. Then, $u_\epsilon$ is a smooth function defined on $\Omega_\epsilon$, and subharmonic on $(\Omega\setminus V)_\epsilon$. Furthermore, since $u$ is of class $C^1$, $u_\epsilon$ and its partial derivative $(u_\epsilon)_{x_i}=\chi_\epsilon*u_{x_i}$ converge to $u$ and $u_{x_i}$ uniformly on any compact subset of $\Omega$ respectively (cf. \cite[\S2.5]{Klimek:1991}). 

Let $\varphi\in C_0^\infty(\Omega)$ with $\varphi\ge 0$. Fix $\delta>0$ with $\delta<\delta_0$ such that ${\rm supp}(\varphi)\subset \Omega_\delta$ for $s=\pm$. When $\delta$ is sufficiently small, $V_t$ ($|t|\le \delta$) intersects $\partial \Omega_\delta$ transversely. If $\epsilon>0$ is sufficiently small, $\overline{\Omega^s_\delta}\subset (\Omega\setminus V)_\epsilon$ for $s=\pm$ because $\Omega^s_\delta$ is relatively compact in $\Omega\setminus V$. Since $\varphi\equiv 0$ around $\partial \Omega_\delta$, by the Green formula,
$$
0\le \int_{\Omega^{s}_\delta}\varphi\Delta u_\epsilon dV=\int_{F^s_\delta}(\varphi (D_n u_\epsilon)-u_\epsilon(D_n\varphi))dS+\int_{\Omega^{s}_\delta}u_\epsilon\Delta \varphi dV
$$
for $s=\pm$,
where $dV$ is the Lebesgue measure on $\mathbb{R}^N$, $dS$ is the surface area element on $F^s_\delta$ and $D_n$ is the outer unit normal derivative on $F^s_\delta$ with respect to $\Omega^s_\delta$ (cf. \cite[\S1.2]{Helms:2009} and \cite[Theorem 16.25]{Lee:2013}).
By letting $\epsilon \to0$, we obtain
\begin{equation}
\label{eq:Stokes_green}
0\le \sum_{s=\pm}\int_{F^s_\delta}(\varphi (D_n u)-u(D_n\varphi))dS+\sum_{s=\pm}\int_{\Omega^{s}_\delta}u\Delta \varphi dV.
\end{equation}

Since $V$ separates $\Omega_\delta^+$ and $\Omega_\delta^-$ and $u$ and $\varphi$ are of class $C^1$ on $\Omega$,
$$
\lim_{\delta\to 0}\int_{F^+_\delta}(\varphi (D_n u)-u(D_n\varphi))dS=-\lim_{\delta\to 0}\int_{F^-_\delta}(\varphi (D_n u)-u(D_n\varphi))dS.
$$
Therefore, we conclude that
$$
0\le  \int_\Omega u\Delta \varphi dV
$$
by letting $\delta\to 0$. This means that $\Delta u\ge 0$ on $\Omega$ in the sense of distribution, and $u$ is subharmonic on $\Omega$ (\cite[Theorem 2.5.8]{Klimek:1991}).
\end{proof}

%
%
%
%
%
%
%
%

%
%
%
%
%




\end{document}